\documentclass[reqno,a4paper,10pt]{amsart}

\usepackage{amssymb, amsthm, amsmath}
\usepackage{fancybox}
\usepackage{color}
\usepackage{version}    
\usepackage{bm}

\newtheorem{theorem}{Theorem}[section]
\newtheorem{definition}[theorem]{Definition}
\newtheorem{lemma}[theorem]{Lemma}
\newtheorem{remark}[theorem]{Remark}
\newtheorem{proposition}[theorem]{Proposition}

\newcommand{\vep}{\varepsilon}
\renewcommand{\d}{\text{\rm d}}
\newcommand{\essinf}{\mathop{\rm ess~inf~}}
\newcommand{\esssup}{\mathop{\rm ess~sup~}}

\makeatletter

\@addtoreset{equation}{section}

\begin{document}

\title[Doubly nonlinear evolution equations]{Subdifferential calculus and 
 doubly nonlinear evolution equations in $L^p$-spaces with variable exponents}

\author{Goro Akagi}
\address[Goro Akagi]{Graduate School of System Informatics, Kobe
University,  1-1 Rokkodai-cho, Nada-ku, Kobe 657-8501, Japan.}
\email{akagi@port.kobe-u.ac.jp}

\author{Giulio Schimperna}
\address[Giulio Schimperna]{Dipartimento di Matematica, Universit\`a di
Pavia, Via Ferrata, 1, I-27100 Pavia, Italy.}
\email{giusch04@unipv.it}

\date{\today}

\thanks{G.~A.~is supported by JSPS KAKENHI Grant Number 25400163,
and by  Hyogo Science and Technology Association, and by Nikko Co.~Ltd.
 Both authors are also supported by the JSPS-CNR bilateral joint
 research project \emph{VarEvol: Innovative Variational methods for
 Evolution PDEs}}

\begin{abstract}
 This paper is concerned with the Cauchy-Dirichlet problem for a doubly
 nonlinear parabolic equation involving variable exponents and provides
 some theorems on existence and regularity of strong
 solutions. In the proof of these results, we also analyze the relations
 occurring between Lebesgue spaces of space-time variables and 
 Lebesgue-Bochner spaces of vector-valued functions,
 with a special emphasis on measurability issues
 and particularly referring to the case of space-dependent 
 variable exponents.
 %
 %
 Moreover, we establish a chain rule for (possibly nonsmooth) 
 convex functionals defined on variable exponent spaces.
 Actually, in such a peculiar functional setting the
 proof of this integration formula is nontrivial and 
 requires a proper reformulation of some basic concepts
 of convex analysis, like those of resolvent, of Yosida 
 approximation, and of Moreau-Yosida regularization.
\end{abstract}



\maketitle
\pagestyle{myheadings}


\section{Introduction}

Nonlinear parabolic equations of the form
\begin{equation}\label{dnp-1}
\beta(\partial_t u) - \Delta u = f \ \mbox{ in } \Omega \times (0,T),
\end{equation}
with a maximal monotone graph $\beta : \mathbb R \to \mathbb R$, a domain
$\Omega$ of $\mathbb R^N$, and a given function $f = f(x,t): \Omega
\times (0,T) \to \mathbb R$, have been studied in various contexts (see,
e.g.,~\cite{Visintin96}). The linear \emph{Laplacian} is often replaced
with nonlinear variants such as the so-called $m$-\emph{Laplacian}
$\Delta_m$ given by
$$
\Delta_m u = \mathrm{div}~ \left(|\nabla u|^{m-2}\nabla u \right), \quad
1 < m < \infty.
$$
In that case, the equation above is called a \emph{doubly nonlinear}
parabolic equation. Very often,  by  setting $u(t) := u(\cdot,t)$,
such a nonlinear parabolic equation is interpreted as an abstract 
evolution equation, i.e., an ordinary differential equation in an
infinite-dimensional space $X$. Namely, one has
\begin{equation}\label{dne}
A(u'(t)) + B(u(t)) = f(t) \ \mbox{ in } X, \quad 0 < t < T,
\end{equation}
with unknown function $u : (0,T) \to X$, two (possibly nonlinear)
operators $A, B$ in $X$, and $f : (0,T) \to X$.
Therefore, it is natural to build the
existence and regularity theory for~\eqref{dne}
in some class of vector-valued functions, like
the Lebesgue-Bochner space $L^p(0,T;X)$. 

Indeed, \eqref{dnp-1} has been studied mostly 
by following two lines: the first one has been originally
developed by Barbu~\cite{B75}, Arai~\cite{Arai} and Senba~\cite{Senba}, 
who analyze \eqref{dne} in a Hilbert space $L^2(0,T;H)$
($H$ denoting here a Hilbert space of functions of
space variables, like for instance $H=L^2(\Omega)$).
Their methods is based on a time differentiation of~\eqref{dne},
which transforms it into another (more tractable) 
type of doubly nonlinear equation,
as well as on some peculiar monotonicity condition, which is, roughly
speaking, formulated by asking that
\begin{equation}\label{cl-hypo1}
\left( Bu - Bv, A(u-v) \right)_H \geq 0,
\end{equation}
where $(\cdot,\cdot)_H$ denotes the inner product of $H$, along with the
homogeneity of $A$.
The other approach has been initiated by Colli-Visintin~\cite{CV} and
Colli~\cite{Colli}, and it relies on the assumption of a power growth for the
maximal monotone operator $A : X \to X^*$ (here $X$ is a Banach space,
for example $X = L^p(\Omega)$), e.g., $p$-power growth given as
\begin{equation}\label{cl-hypo2}
|Au|_{X^*}^{p'} \leq C(|u|_X^p + 1), \quad c_0 |u|_X^p \leq \langle Au,
u \rangle_X + C 
\end{equation}
with constants $c_0 > 0$, $C \geq 0$ and $p' := p/(p-1)$.
Particularly, in~\cite{Colli}, 
equation~\eqref{dne} is analyzed in the Banach space $L^p(0,T;X)$.
For other results on doubly nonlinear equations of the form \eqref{dne},
the reader is referred 
to~\cite{AiziYan,GO3,GradFlow,AsFrK,BlaDam,MiTh,MiRo,MiRoSa,Roubicek,SSS,Segatti,S-ken}
and references therein.

In this paper, we are concerned with the following Cauchy-Dirichlet
problem for doubly nonlinear parabolic equations with \emph{variable
exponents}:
\begin{alignat}{3}
 |\partial_t u|^{p(x)-2} \partial_t u - \Delta_{m(x)} u &= f(x,t)
\quad &\mbox{ in }& \ Q := \Omega \times (0,T), \label{pde}\\
u &= 0 \quad &\mbox{ on }& \ \partial \Omega \times (0,T), \label{bc}\\
u(\cdot,0) &= u_0 &\mbox{ in }& \ \Omega, \label{ic}
\end{alignat}
where $1 < p(x), m(x) < \infty$ are variable exponents and
$\Delta_{m(x)}$ stands for the \emph{$m(x)$-Laplacian} given by
$$
\Delta_{m(x)} u
= \mathrm{div}~ \left(
|\nabla u|^{m(x)-2} \nabla u
\right).
$$
The constant exponent case, i.e., $p(x) \equiv p$, $m(x) \equiv m$, can be
treated within the classical frame mentioned above by making
appropriate assumptions. However, the variable exponent case presents
a number of peculiarities which do not permit to apply the standard
theory. Indeed, due to the $x$-dependence of
variable exponents, it is difficult to check the monotonicity
condition \eqref{cl-hypo1} used in the first approach (e.g.,~\cite{Arai})
without taking extra assumptions (e.g., a smooth dependence
of the variable exponent with respect to the space variables).
On the other hand, differently from \cite{CV} and~\cite{Colli},
here the operator $A : u \mapsto |u(\cdot)|^{p(\cdot)-2}u(\cdot)$ 
with the variable exponent $p(x)$ does not satisfy any 
mathematically tractable $p$-growth 
condition \eqref{cl-hypo2}, because the 
growth order is inhomogeneous over $\Omega$. 
For all these reasons, a proper functional setting for 
addressing Problem~\eqref{pde}--\eqref{ic} seems still to be lacking,
and, actually, the main aim of this paper consists exactly in
the construction of such a framework.

In order to understand which are our main ideas, we have 
to focus on the boundedness of the operator $\mathcal A : u \mapsto |u|^{p(x)-2}u$ 
from $L^{p(x)}(Q)$ to the dual space $L^{p'(x)}(Q) = (L^{p(x)}(Q))^*$, 
where $L^{p(x)}(Q)$ and $L^{p'(x)}(Q)$ stand for variable exponent 
Lebesgue spaces of space-time variables over 
the set $Q = \Omega \times (0,T)$ with $p'(x) := p(x)/(p(x) - 1)$. 
Then, it is easy to notice that this boundedness cannot be
formulated in terms of any Lebesgue-Bochner space of vector-valued
functions with no loss of integrability, due to the presence of 
variable exponents. To say it in simple words, the identification 
$L^{p(x)}(Q) \sim L^{p(x)}(0,T;L^{p(x)}(\Omega))$,
which is standardly used for constant exponents $p\in[1,\infty)$,
turns out to be meaningless in the variable-exponent setting; actually,
once one tries to embed $L^{p(x)}(Q)$ into some vector-valued function 
space of the time variable, a loss of integrability occurs. 
For this reason, we are obliged to address equation \eqref{pde}
by mainly working in the space $L^{p(x)}(Q)$, which
plays a critical role as far as we need to exploit 
the fine properties of the operator $\mathcal A$. 
This functional setting forces us to pay attention 
to the different \emph{measures} and \emph{measurability}
concepts characterizing Lebesgue spaces of space-time
variables and Lebesgue-Bochner spaces of vector-valued
functions, particularly in the variable-exponent case.
In addition to this, the monotone structure
of the nonlinear operators appearing in equation~\eqref{pde}
has to be properly managed in the setting of the space 
$L^{p(x)}(Q)$. Indeed, while for constant exponents
any (maximal) monotone operator acting on $L^p(\Omega)$ 
can be extended to the time-dependent space $L^p(Q)$
(or, equivalently, to $L^p(0,T;L^p(\Omega))$)
in a  straightforward  way, here this procedure is far
from being obvious because space and time variables
cannot be ``decoupled'' in the definition of $L^{p(x)}(Q)$.
In order to overcome this problem, we have to revise
some concepts in the theory of monotone operators 
and of subdifferentials and adapt them to the variable
exponent setting. In particular, we need to properly
modify the notions of Yosida approximation for monotone
operators and of Moreau-Yosida regularization for convex
functionals. This permits us to prove a chain rule for 
subdifferentials, extending the classical result
\cite[Lemme~3.3, p.~76]{HB1}) to the space $L^{p(x)}(Q)$
(cf.~Prop.~\ref{P:chain} below). This chain rule will
play a crucial role in the existence proof for
\eqref{pde}--\eqref{ic}.

It is worth noting that the functional framework 
and the convex analysis tools developed in the present paper 
could also be applied to more general classes of doubly nonlinear parabolic
equations, like for instance
$$
\beta(x,\partial_t u) - \mathrm{div} ~{\bf a}(x,\nabla u) = f,
$$
with $x$-dependent maximal monotone graphs $\beta(x,\cdot)$ in $\mathbb
R$ and ${\bf a}(x,\cdot)$ in $\mathbb R^N$ under $p(x)$- and
$m(x)$-growth conditions on $\beta(x,\cdot)$ and ${\bf a}(x,\cdot)$,
respectively, at each point $x \in \Omega$.
%

Prior to stating main results, let us exhibit our basic assumptions (H):
\begin{align}
 m \in \mathcal P_{\log}(\Omega), \quad
 p \in \mathcal P(\Omega), \quad
 1 < p^-, m^-, p^+, m^+ < \infty, \label{H1}\tag{H1}\\
 \displaystyle \essinf_{x \in \Omega} \left( m^*(x) - p(x) \right) > 0,
 \label{H2}\tag{H2}\\
 f \in L^{p'(x)}(Q), 
 \quad u_0 \in W^{1,m(x)}_0(\Omega), \label{H3}\tag{H3}
\end{align}
where $\mathcal P_{\log}(\Omega)$ (resp., $\mathcal P(\Omega)$) stands
for the set of log-H\"older continuous (resp., measurable) exponents $1
\leq p(x) \leq \infty$ over $\Omega$, $p^- := \essinf p(x)$, $p^+
:= \esssup p(x)$, $m^\pm$ are defined analogously for $m(x)$, and
$m^*(x) := (Nm(x))/(N-m(x))_+$ (see \S \ref{S:P} below for more details).

We are concerned with solutions of \eqref{pde}--\eqref{ic} defined in
the following sense:
\begin{definition}[Strong solutions]\label{D:ssol}
 We call $u \in L^{p(x)}(Q)$ a \emph{strong solution} of
 \eqref{pde}--\eqref{ic} in $Q$ whenever the following conditions hold
 true\/{\rm :}
\begin{enumerate}
 \item $t \mapsto u(\cdot,t)$ is continuous with values in
       $L^{p(x)}(\Omega)$ on $[0,T]$ and is weakly
       continuous with values in $W^{1,m(x)}_0(\Omega)$ on $[0,T]$,
 \item $\partial_t u \in L^{p(x)}(Q)$, $\Delta_{m(x)} u \in
       L^{p'(x)}(Q)$,
 \item equation \eqref{pde} holds for a.e.~$(x,t) \in Q$, 
 \item the initial condition \eqref{ic} is satisfied for a.e.~$x \in \Omega$.
\end{enumerate} 
\end{definition}
Now, our result reads
\begin{theorem}[Existence of strong solutions]\label{T:ex}
 Assume {\rm (H)}. Then the Cauchy-Dirichlet problem
 \eqref{pde}--\eqref{ic} admits (at least) one strong
 solution $u$ in the sense of Definition {\rm \ref{D:ssol}}.
\end{theorem}
In the case when the forcing term $f$ is more regular,
namely
\begin{equation}\label{f-regu}
 t \partial_t f \in L^{p'(x)}(Q),
\end{equation}
we can also prove parabolic regularization properties 
of strong solutions:
\begin{theorem}[Time-regularization of strong solutions]\label{T:regu}
 Assume \eqref{f-regu} together with {\rm (H)}. Then, the
 Cauchy-Dirichlet problem \eqref{pde}--\eqref{ic} admits a strong
 solution $u$ in the sense of Definition~{\rm \ref{D:ssol}},
 which additionally satisfies
 \begin{equation}\label{regu}
  \esssup_{t \in (\delta,T)} \|\partial_t u(\cdot,t)\|_{L^{p(x)}(\Omega)} <
   \infty 
   \quad \mbox{ and } \quad
   \esssup_{t \in (\delta,T)} \|\Delta_{m(x)}
   u(\cdot,t)\|_{L^{p'(x)}(\Omega)} < \infty 
\end{equation}
 for any $\delta \in (0,T)$.
\end{theorem}
\noindent
The remainder of the paper is organized as follows:
in Section \ref{S:P}, we summarize some preliminary material on convex
analysis and variable exponent Lebesgue and Sobolev spaces to be used
later. In Section \ref{S:red}, we reduce \eqref{pde}--\eqref{ic} to a
doubly nonlinear evolution equation and discuss its representation in
a Lebesgue space of space-time variables as well as a pointwise (in
time) one. Moreover, we also provide a summary of the relations
occurring between Lebesgue spaces of space-time variables and 
Lebesgue-Bochner spaces of vector-valued functions (see
Proposition \ref{P:ident}). Section \ref{S:th1} is devoted to a proof of
Theorem \ref{T:ex}. Our argument basically relies on a time-discretization
and limiting procedure. In particular, a chain-rule for convex functionals 
in a mixed frame turns out to play a crucial role. To prove this chain-rule, 
we introduce some modified definitions for
\emph{resolvent}, \emph{Yosida approximation} and
\emph{Moreau-Yosida regularization}, which, compared to the standard ones,
are more suitable for working in the variable-exponent setting.
In Section \ref{S:th2}, we give a proof of
Theorem \ref{T:regu} by performing a second energy estimate in a
discrete level. Finally, in the Appendix, we present a survey
of the theory of Lebesgue-Bochner spaces and see how this
theory can be extended to the variable exponent case. In
particular, we give a proof of 
Proposition~\ref{P:ident} of \S~\ref{Ss:space}. 

\noindent
{\bf Notation.}
We write $Q = \Omega \times (0,T)$.
For vector-valued functions $u : (0,T) \to X$, we denote by $u'$ the
$X$-valued derivative of $u$ in time. For $u : Q \to \mathbb R$, 
the partial derivative of $u$ in time is denoted by $\partial_t u$.


\section{Preliminaries}\label{S:P}

This section is devoted to recall some preliminary results on convex analysis
and on Lebesgue and Sobolev spaces with variable exponents.


\subsection{Convex analysis}\label{Ss:CoAn}

Let $E$ be a reflexive Banach space with a norm $|\cdot|_E$, a dual space $E^*$
(with a norm $|\cdot|_{E^*}$), and a duality pairing $\langle \cdot,
\cdot \rangle_E$ (or $\langle \cdot, \cdot \rangle$ for short) between
$E$ and $E^*$. Let $\varphi : E \to (-\infty,\infty]$ be a proper (i.e.,
$\varphi \not\equiv \infty$), lower semicontinuous convex function with
the \emph{effective domain}
$$
  D(\varphi) := \left\{ u \in E \colon \varphi(u) < \infty \right\}.
$$
The \emph{subdifferential operator} $\partial \varphi : E \to 2^{E^*}$
associated with $\varphi$ is defined by
$$
\partial \varphi(u) := \left\{ \xi \in E^* \colon \varphi(v) - \varphi(u) \geq
\langle \xi, v - u \rangle \quad \mbox{ for all } \ v \in D(\varphi) \right\},
$$
for $u \in D(\varphi)$, with the domain
$D(\partial \varphi) := \{u \in D(\varphi) \colon \partial \varphi(u) \neq
\emptyset\}$.  It is well known that the subdifferential 
of any convex functional is a maximal monotone 
operator in $E \times E^*$. Furthermore,
$\varphi$ is said to be \emph{Fr\'echet differentiable} in $E$ if, for
any $u \in E$, there exists $\xi_u \in E^*$ such that
$$
\left| \dfrac{\varphi(u+v)-\varphi(u)-\langle \xi_u, v
\rangle}{|v|_E}\right| \to 0, \quad \mbox{ whenever } \quad 
v \to 0 \ \mbox{ strongly in } E.
$$
Then $\d \varphi : E \to E^*$,
$\d \varphi: u \mapsto \xi_u$ is called a \emph{Fr\'echet
derivative} of $\varphi$. In particular, if $\varphi$ is convex and Fr\'echet
differentiable in $E$, then $\partial \varphi = \d \varphi$.

Moreover, the \emph{convex conjugate} $\varphi^* : E^* \to (-\infty,\infty]$ of
$\varphi$ is defined as
$$
\varphi^*(\xi):= \sup_{u \in E} \left(
\langle \xi, u \rangle - \varphi(u)
\right) \quad \mbox{ for } \ \xi \in E^*.
$$
It particularly holds that $\partial \varphi^* = (\partial
\varphi)^{-1}$, that is, $\xi \in \partial \varphi(u)$ if and only if $u
\in \partial \varphi^*(\xi)$.


\subsection{Variable exponent Lebesgue and Sobolev spaces}\label{Ss:varExp}

In this subsection, we briefly review the theory of Lebesgue and
Sobolev spaces with variable exponent to be used later.
The reader is referred to~\cite{DHHR11} for a more detailed
survey of this field. Let $\mathcal O$ be a domain in $\mathbb{R}^N$. 
We denote by $\mathcal P(\mathcal O)$ the set of all measurable functions $p
: \mathcal O \to [1,\infty]$. For $p \in \mathcal P(\mathcal O)$, we write
$$
 p^+ := \esssup_{x \in \mathcal O} p(x),
  \quad p^- := \essinf_{x \in \mathcal O} p(x).
$$
Throughout this subsection, we assume that $p \in \mathcal P(\mathcal O)$.
Then, for $p^+<+\infty$, the Lebesgue space with a variable exponent $p(x)$ 
is defined as follows:
$$
L^{p(x)}(\mathcal O) 
 := \bigg\{
     u : \mathcal O \to \mathbb{R} \colon \mbox{ measurable in $\mathcal O$ and }
     \int_{\mathcal O} |u(x)|^{p(x)} \d x < \infty
    \bigg\}
$$
with a Luxemburg-type norm
$$
\|u\|_{L^{p(x)}(\mathcal O)}
  := \inf \left\{
	   \lambda > 0 \colon \int_\mathcal O \left|
				      \frac{u(x)}{\lambda}
				     \right|^{p(x)} \d x \leq 1
	  \right\}.
$$
Then $L^{p(x)}(\mathcal O)$ is a special sort of Musielak-Orlicz space
(see~\cite{Musielak}) and is sometimes called \emph{Nakano space}. 
For $p^+ < \infty$, the dual space of $L^{p(x)}(\mathcal O)$ is
identified with $L^{p'(x)}(\mathcal O)$ with the dual variable
exponent $p' \in \mathcal P$ given by
$$
\dfrac 1 {p(x)} + \dfrac 1 {p'(x)} = 1 \quad \mbox{ for a.e.~} x \in \mathcal O,
$$
where we write $1/\infty = 0$. In the case when $p^+=+\infty$ the 
above definition can be adapted with minor changes (see, e.g.,
\cite[Chap.~3]{DHHR11}). 

H\"older's inequalities also hold for variable exponent Lebesgue spaces
(cf.~\cite[Lemma 3.2.20]{DHHR11}):
\begin{proposition}[H\"older's inequality]\label{P:Holder}
For $s,p,q \in \mathcal P(\mathcal O)$, it holds that
$$
\|fg\|_{L^{s(x)}(\mathcal O)} \leq 2 \|f\|_{L^{p(x)}(\mathcal O)} \|g\|_{L^{q(x)}(\mathcal O)}
\quad \mbox{ for all } \ f \in L^{p(x)}(\mathcal O),\
g \in L^{q(x)}(\mathcal O),
$$
provided that
$$
\dfrac 1 {s(x)} = \dfrac 1 {p(x)} + \dfrac 1 {q(x)}
\quad \mbox{ for a.e. }x \in \mathcal O.
$$
In particular, if $\mathcal O$ is bounded and $p(x) \leq q(x)$ 
for a.e.~$x \in \mathcal O$, then $L^{q(x)}(\mathcal O)$ is continuously 
embedded in $L^{p(x)}(\mathcal O)$.
\end{proposition}
The following proposition plays an important role to establish energy
estimates (see, e.g., Theorem 1.3 of~\cite{FZ01} for a proof).
\begin{proposition}\label{P:px-est}
Let $p^+<\infty$. Then, we have
$$
\sigma_{p(\cdot)}^-(\|w\|_{L^{p(x)}(\mathcal O)}) \leq
 \int_\mathcal O|w(x)|^{p(x)} \d x
\leq \sigma_{p(\cdot)}^+(\|w\|_{L^{p(x)}(\mathcal O)})
\quad \mbox{ for all } \ w \in L^{p(x)}(\mathcal O)
$$
with the strictly increasing functions
$$
\sigma_{p(\cdot)}^-(s) := \min\{s^{p^-}, s^{p^+}\}, \quad
\sigma_{p(\cdot)}^+(s) := \max\{s^{p^-}, s^{p^+}\} \quad \mbox{ for } s \geq 0.
$$
\end{proposition}
We next define variable exponent Sobolev spaces $W^{1,p(x)}(\mathcal O)$ 
as follows:
$$
W^{1, p(x)}(\mathcal O) := \left\{
				u \in L^{p(x)}(\mathcal O) \colon
				\frac{\partial u}{\partial x_i} \in L^{p(x)}(\mathcal O) 
				\quad \mbox{ for all } \ i = 1,2,\ldots, N
			       \right\}
$$
with the norm 
$$
\|u\|_{W^{1,p(x)}(\mathcal O)} 
 :=
 \left(
  \|u\|_{L^{p(x)}(\mathcal O)}^2 + \|\nabla u\|_{L^{p(x)}(\mathcal O)}^2
 \right)^{1/2}, 
$$
where $\|\nabla u\|_{L^{p(x)}(\mathcal O)}$ denotes the
$L^{p(x)}(\mathcal O)$-norm of $|\nabla u|$. Furthermore, let 
$W^{1, p(x)}_0 (\mathcal O)$ be the closure of
$C^\infty_0(\mathcal O)$ in $W^{1,p(x)}(\mathcal O)$. Here we note that,
usually, the space $W^{1,p(x)}_0(\mathcal O)$ is defined in a slightly different
way for the variable exponent case. However, both definitions are
equivalent under the regularity assumption \eqref{log-conti} 
given below.

The following proposition is concerned with the uniform convexity of
$L^{p(x)}$- and $W^{1,p(x)}$-spaces.
\begin{proposition}[\cite{DHHR11}]
If $p^+ < \infty$, then $L^{p(x)}(\mathcal O)$ is a separable Banach space.
 If  $p^- > 1$ and $p^+ < \infty$, 
then $L^{p(x)}(\mathcal O)$ and
 $W^{1,p(x)}(\mathcal O)$ are uniformly convex.
Hence they are reflexive.
\end{proposition}

Let us exhibit the Poincar\'e and Sobolev inequalities. 
To do so, we introduce the {\it log-H\"older
condition}:
\begin{equation}\label{log-conti}
 |p(x) - p(x')| \leq \frac{A}{\log(e + 1/|x-x'|)}
  \quad \mbox{ for all } x,x' \in \mathcal O
\end{equation}
with some constant $A > 0$ (see~\cite{DHHR11}). 
This condition is weaker than the H\"older continuity of $p$ over
$\overline{\mathcal O}$ and it implies $p \in C(\overline{\mathcal O})$
and $p^+ < \infty$. We denote by $\mathcal P_{\log}(\mathcal O)$ the set of all
$p \in \mathcal P(\mathcal O)$ satisfying the log-H\"older condition
\eqref{log-conti}.

Then the following properties hold:
\begin{proposition}[\cite{DHHR11}]\label{P:SP}
Let $\mathcal O$ be a bounded domain in $\mathbb R^N$ with smooth boundary
 $\partial \mathcal O$ and let $p \in \mathcal P_{\log}(\mathcal O)$.
\begin{enumerate}
 \item[\rm (i)] There exists a constant $C \geq 0$ such that
       $$
       \|w\|_{L^{p(x)}(\mathcal O)} \leq C \|\nabla w\|_{L^{p(x)}(\mathcal O)}
       \quad \mbox{ for all } \ w \in W^{1,p(x)}_0(\mathcal O).
       $$
       In particular, the space $W^{1,p(x)}_0(\mathcal O)$
       has a norm $\|\cdot\|_{1,p(x)}$ given by
       $$
       \|w\|_{1,p(x)} := \|\nabla w\|_{L^{p(x)}(\mathcal O)} \quad 
       \mbox{ for } \ w \in W^{1,p(x)}_0(\mathcal O),
       $$
       which is equivalent to $\|\cdot\|_{W^{1,p(x)}(\mathcal O)}$.
 \item[\rm (ii)] Let $q : \mathcal O \to [1,\infty)$ be a measurable and
	      bounded function and suppose that 
	      $$
	      q(x) \leq p^*(x) := Np(x)/(N - p(x))_+
	      \quad \mbox{ for a.e. } x \in \mathcal O,
	      $$
	      where $(s)_+ := \max\{s,0\}$ for $s \in \mathbb R$.
	      Then $W^{1,p(x)}(\mathcal O)$ is continuously embedded 
	      in $L^{q(x)}(\mathcal O)$. 
	      
	      In addition, assume that
	      $$
	      \essinf_{x \in \mathcal O} \big(
	      p^*(x) - q(x)
	      \big) > 0.
	      $$
	      Then the embedding $W^{1,p(x)}(\mathcal O) \hookrightarrow 
	      L^{q(x)}(\mathcal O)$ is compact.
\end{enumerate}
\end{proposition}

\begin{remark}
{\rm
In~\cite{MOSS}, it is proved that the embedding
 $W_0^{1,p(x)}(\mathcal O) \hookrightarrow L^{q(x)}(\mathcal O)$ is
 compact when $p^*(x)$ coincides with $q(x)$ on some thin part of $\mathcal O$ 
 and the difference between the two variable exponents is appropriately
 controlled on the other part (see also~\cite{KS}).
}
\end{remark}

Finally we give a variant of Young's inequality with variable exponents.
Let $p \in \mathcal P(\mathcal O)$ with $p^+<\infty$.
For any $\vep > 0$, there exists a constant $C_\vep \geq 0$ independent
of $x$ such that
\begin{equation}\label{Young}
ab \leq \vep a^{p(x)} + C_\vep b^{p'(x)}
\quad \mbox{ for all } a, b \geq 0 \ \mbox{ and \ for a.e.} \ 
x \in \mathcal O.
\end{equation}
Indeed, let $\delta \in (0,1)$ be arbitrarily given. Then, from the standard
form of Young's inequality, we have
\begin{align*}
 ab 
 = (\delta a) \dfrac b \delta 
 &\leq \dfrac{\delta^{p(x)}}{p(x)} a^{p(x)} + \dfrac 1 {p'(x)
 \delta^{p'(x)}} b^{p'(x)} \\
 &\leq \dfrac{\delta^{p^-}}{p^-} a^{p(x)}
+ \dfrac{1}{(p^+)' \delta^{(p^-)'} } b^{p'(x)}.
\end{align*}
For each $\vep > 0$, take a constant $\delta_\vep \in (0,1)$ such that
$\vep \geq \delta_\vep^{p^-}/p^-$. Then \eqref{Young} follows with a
constant $C_\vep := ((p^+)' \delta_\vep^{(p^-)'})^{-1} \geq 0$.


\section{Reduction to an abstract evolution equation}\label{S:red}


\subsection{Setting of spaces and potentials}\label{Ss:space}

We set $V = L^{p(x)}(\Omega)$ and $X = W^{1,m(x)}_0(\Omega)$ with norms
$\|u\|_V := \|u\|_{L^{p(x)}(\Omega)}$ and $\|u\|_X := \|\nabla
u\|_{L^{m(x)}(\Omega)}$, respectively. Moreover, we write
$$
\langle v, u \rangle_V = \int_\Omega u(x) v(x) ~\d x
\quad \mbox{ for all } \ u \in V, \ v \in V^* = L^{p'(x)}(\Omega).
$$
By assumption \eqref{H2} along with Proposition \ref{P:SP}, it follows that
$$
X \hookrightarrow V \quad \mbox{ and } \quad V^* \hookrightarrow X^*
$$
where the embeddings are continuous and, in view of
\cite[Thm.~3.4.12, p.~90]{DHHR11}, they are also dense.

Define functionals $\psi$ and $\phi$ on $V$ by
$$
\psi(u) := \int_\Omega \dfrac 1 {p(x)} |u(x)|^{p(x)} \d x
	    \quad \mbox{ for } \ u \in V
$$
and
$$
\phi(u) := \begin{cases}
	       \displaystyle \int_\Omega \dfrac 1 {m(x)} |\nabla
	       u(x)|^{m(x)} \d x
	       &\mbox{ if } \ u \in X,\\
	       \infty &\mbox{ if } \ u \in V \setminus X.
	      \end{cases}
$$
Here and henceforth, we use $\partial_\Omega$ for the subdifferential
in $V=L^{p(x)}(\Omega)$ and $\partial_Q$ for the subdifferential in
$L^{p(x)}(Q)$ when any confusion may arise.

Then, $\psi$ is Fr\'echet differentiable on $V$ 
and we  find  that $\partial_\Omega \psi(u) = 
\d \psi(u) = |u|^{p(x)-2}u$ for any $u\in V$.

On the other hand,  $\phi$ is proper, lower semicontinuous and
convex in $V$. The lower semicontinuity can be proved in a standard way
(see Lemma 3.2 of~\cite{AM13}). Moreover, it holds that 
$\partial_\Omega \phi(u) = - \Delta_{m(x)} u$ with the domain
$$
  D(\partial_\Omega \phi) = \left\{ u \in X : 
   - \Delta_{m(x)} u \in L^{p'(x)}(\Omega) \right\}   
$$
incorporating the boundary condition \eqref{bc} in the 
sense of traces. Actually, the restriction $\phi|_X$ of 
$\phi$ to $X$ is also Fr\'echet differentiable in
$X$ with the representation $\d (\phi|_X)(u) = -\Delta_{m(x)} u$; 
moreover, $\partial_\Omega \phi(u) = \d (\phi|_X) (u)$ for any $u \in
D(\partial_\Omega \phi)$. 

The above defined operators permit to reduce \eqref{pde}--\eqref{ic} 
into the following doubly nonlinear evolution equation:
\begin{align}
 \partial_\Omega \psi(u'(t)) + \partial_\Omega \phi(u(t)) = Pf(t) \
 \mbox{ in } \ V^*, \quad 0 < t < T, \label{ee}\\
 u(0) = u_0.
\end{align}
The operator $P$ represents pointwise in time evaluation for 
functions of space-time variables and will be more precisely defined
later on. Here, we notice that $Pf(t) := f(\cdot,t) \in V^*$ 
for a.e.~$t\in(0,T)$ (this follows from (H3) and 
Proposition \ref{P:ident} below).

As mentioned in the Introduction, we shall work in a mixed frame of
Lebesgue-Bochner space $L^p(0,T;L^p(\Omega))$ and Lebesgue space
$L^p(Q)$ with $Q = \Omega \times (0,T)$.
However, these classes of spaces are originally 
defined in a different way and their identification is delicate,
particularly in the variable-exponent setting, 
in view of the different types of \emph{measures} involved.
In the Appendix we will present a review of the underlying
theory by emphasizing the additional difficulties occurring
in the variable exponent case. A crucial role will be played 
by the \emph{pointwise evaluation}
operator $P:L^1(Q) \to L^1(0,T;L^1(\Omega))$ (as $|\Omega|<\infty$,
$L^1$ is the largest space), defined by
$Pu(t) := u(\cdot,t)$ for $t \in (0,T)$,
which permits to pass
from Lebesgue functions of space-time variables to 
Lebesgue-Bochner vector-valued functions. Its properties
are summarized in the following proposition, whose
proof is postponed to the Appendix, where an extended survey
of the properties of $P$ is presented. Here and henceforth, we
simply write $Pu(t)$ and $P^{-1}u(x,t)$ instead of $(Pu)(t)$ and
$(P^{-1}u)(x,t)$, respectively. 
\begin{proposition}\label{P:ident}
 For any constant exponent $1 \leq p < \infty$ and variable one 
 $p(x)$ with $1 \leq p^- \leq p^+ < \infty$, the following 
 {\rm (i)--(iv)} hold true\/{\rm :}
\begin{enumerate}
 \item The operator $P$ is a linear, bijective, isometric mapping from
       $L^p(Q)$ to $L^p(0,T;L^p(\Omega))$. Furthermore, if $u \in
       L^{p(x)}(Q)$, then $Pu \in L^{p^-}(0,T;L^{p(x)}(\Omega))$.
 \item The inverse $P^{-1} : L^p(0,T;L^p(\Omega)) \to L^p(Q)$ is
       well-defined, and for each $u = u(t) \in L^p(0,T;L^p(\Omega))$,
       it holds that $u(t) = P^{-1}u (\cdot,t)$ for a.e.~$t \in (0,T)$.
 \item If $u \in L^{p(x)}(Q)$ with $\partial_t u \in L^{p(x)}(Q)$,
       then $Pu$ belongs to the space
       $W^{1,p^-}(0,T;L^{p(x)}(\Omega))$ and $(Pu)' = P(\partial_t u)$.
 \item If $u \in W^{1,p}(0,T;L^p(\Omega))$, then $\partial_t (P^{-1}u)$
       belongs to $L^p(Q)$ and coincides with $P^{-1}(u')$, where $u'$
       denotes the derivative of $u : (0,T) \to V$ in time. 
\end{enumerate}
\end{proposition}
To be more precise (c.f., e.g., (ii) above), 
for each exponent $p$ (or $p(x)$), we should use a different
notation of the operators defined above.
However, for simplicity, we shall always
write $P$ and $P^{-1}$ regardless of $p$.

Set $\mathcal V := L^{p(x)}(Q)$ with
the norm $\|u\|_{\mathcal V} := \|u\|_{L^{p(x)}(Q)}$ and $\langle u, v
\rangle_{\mathcal V} := \iint_Q u(x,t)v(x,t) ~\d x \, \d t$
whenever $u\in \mathcal V$ and $v \in \mathcal V^* = L^{p'(x)}(Q)$.
Define functionals $\Psi$ and $\Phi$ on $\mathcal V$ as
$$
\Psi(u) := \iint_Q \dfrac 1 {p(x)} |u(x,t)|^{p(x)} \d x \, \d t
= \int^T_0 \psi(Pu(t)) ~\d t,
$$
where the latter equality follows from Fubini's lemma and the fact $u
\in L^{p(x)}(Q)$, and
$$
\Phi(u) := \begin{cases} 
	    \displaystyle \int^T_0 \phi(Pu(t)) ~\d t
	    &\mbox{ if } \ Pu(t) \in X \ \mbox{
	    for a.e. } t \in (0,T), \ t \mapsto \phi(Pu(t)) \in L^1(0,T),\\
	    \infty &\mbox{ otherwise} 
	   \end{cases}
$$
for $u \in \mathcal V$. Then $\Psi$ is Fr\'echet differentiable and
convex in $\mathcal V$ (hence $D(\Psi) = \mathcal V$) and $\Phi$ is
proper, lower semicontinuous and convex in $\mathcal V$ with 
$D(\Phi) = \{u \in \mathcal V \colon Pu(t) \in
X \ \mbox{ for a.e.~} t \in (0,T) \mbox{ and } t \mapsto
\phi(Pu(t)) \in L^1(0,T) \}$. To prove the lower semicontinuity of
$\Phi$ on $\mathcal V$, it suffices to check it on a larger space
$L^1(0,T;V)$ which continuously embeds $\mathcal V$ (see (ii) of Lemma
\ref{L:rel-sp} below) by using the lower semicontinuity of $\phi$ in
$V$. The details, quite standard, are given in Appendix~\ref{app:B} below.
The subdifferential operator $\partial_Q \Psi : \mathcal V \to \mathcal
V^*$ of $\Psi$ is formulated as
$$
\partial_Q \Psi(u) := \left\{
\xi \in \mathcal V^* \colon
\Psi(v)-\Psi(u) \geq \iint_Q \xi (v - u) ~\d x \, \d t
\quad \mbox{ for all } \ v \in D(\Psi)
\right\}
$$
with domain $D(\partial_Q \Psi) := \{u \in D(\Psi) \colon \partial_Q \Psi
(u) \neq \emptyset\}$,
and $\partial_Q \Phi$ can be also defined analogously.

In the constant exponent case, a similar extension of convex functionals
onto Lebesgue-Bochner spaces (e.g., $L^p(0,T;V)$) is a standard issue.
On the other hand, in our case $\mathcal V = L^{p(x)}(Q)$, the 
given extensions $\Phi$ and $\Psi$ of $\phi$ and $\psi$
do not correspond to those provided by the standard theory. 
Correspondingly, 
some basic properties of these functionals (like, e.g., 
subdifferentials, or regularizations) need to be properly
analyzed.

The following relations will be frequently
used in the sequel:
\begin{enumerate}
 \item For $u \in \mathcal V$ and $\xi \in \mathcal V^*$, 
       \begin{align*}
       [u,\xi] \in \partial_Q \Phi \quad \mbox{ if and only if } \quad
       [Pu(t),P\xi(t)] \in \partial_\Omega \phi \ \mbox{ for a.e. }  
       t \in (0,T), \\ 
	\mbox{ i.e., }
       - \Delta_{m(x)} u(x,t) = \xi(x,t) \ \mbox{ for a.e. } 
       (x,t) \in Q.
       \end{align*}
 \item For $u \in \mathcal V$ and $\eta \in \mathcal V^*$,
       \begin{align*}
       [u,\eta] \in \partial_Q \Psi \quad \mbox{ if and only if } \quad
       [Pu(t),P\eta(t)] \in \partial_\Omega \psi \ \mbox{ for a.e. }  
       t \in (0,T), \\ 
	\mbox{ i.e., }
       |u(x,t)|^{p(x)-2}u(x,t) = \eta(x,t) \ \mbox{ for a.e. } 
       (x,t) \in Q.
       \end{align*}
\end{enumerate}
The above properties (i)--(ii) will be proved
in the next subsection, where, actually, 
more general results will be presented.
Finally, on account of the previous discussion, we can restate 
equation \eqref{ee} in Lebesgue spaces of space-time variables 
as follows:
\begin{equation}\label{ee-2}
 \partial_Q \Psi(\partial_t (P^{-1}u)) + \partial_Q \Phi(P^{-1}u) = f \
  \mbox{ in } \mathcal V^*.
\end{equation}
Then $\hat u := P^{-1}u$ corresponds to a strong solution of
\eqref{pde}--\eqref{ic} as in Definition \ref{D:ssol}, provided that
$\hat u$ enjoys sufficient regularity.


\subsection{Representation of subdifferential operators associated with
  variable exponents}\label{Ss:repre}

In this section, we set $V=L^{p(x)}(\Omega)$ with
$1<p^-\le p^+<\infty$ (then $V$ turns out to be a reflexive and 
separable Banach space) and let $p'(x)$ be the (pointwise) 
conjugate exponent of $p(x)$, that is, $p'(x) := p(x)/(p(x)-1)$. We also
let $\varphi:V\to [0,\infty]$ be proper (i.e., $D(\varphi) \neq
\emptyset$; for simplicity, we assume $D(\varphi) \ni 0$),
convex and lower semicontinuous (note that here $\varphi$ is 
a \emph{generic} functional, non-necessarily corresponding
to the functional $\phi$ of our equation). Then,
we define a functional $\Phi$ on $\mathcal
V := L^{p(x)}(Q)$ by setting
\begin{equation}
 \label{Phi}
   \Phi(u):= \begin{cases}
	      \displaystyle \int^T_0 \varphi(Pu(t))\,\d t \ &\mbox{ if } \
	      \varphi(Pu(\cdot)) \in L^1(0,T),\\
	      \infty &\mbox{ otherwise.}
	     \end{cases}
\end{equation}
Then $\Phi$ is proper, lower semicontinuous and convex.
As expected, we have
\begin{lemma}\label{lemmaext}
For $u \in \mathcal V$, $\xi \in \mathcal V^*$ with $1 < p^- \leq p^+ <
 \infty$, the following property holds\/{\rm :}
 \begin{equation}\label{equiv}
   \xi \in \partial_Q \Phi(u) \quad \mbox{ if and only if } \quad
    P\xi(t) \in \partial_\Omega \varphi(Pu(t))~\,\text{for a.e.~}\,t\in(0,T).
 \end{equation}
\end{lemma}
\begin{proof}
Define the operator $A_{\text{ext}}:\mathcal V\to \mathcal V^*$ as
\begin{equation}\label{co11}
  \xi \in A_{\text{ext}}(u) \quad
   \stackrel{\text{define}}\Longleftrightarrow \quad
   P\xi(t) \in \partial_\Omega \varphi(Pu(t))~\,\text{for a.e.~}\,t\in(0,T),
\end{equation}
where a function $u\in \mathcal V$ belongs to
the domain $D(A_{\text{ext}})$ whenever
$Pu(t) \in D(\partial_\Omega \varphi)$ for a.e.~$t \in (0,T)$
and, moreover, there exists a function $\xi \in \mathcal V^*$ 
such that $P\xi(t) \in \partial_\Omega \varphi(Pu(t))$ 
for a.e.~$t \in (0,T)$.
Then one readily verifies by integration in time and Fubini's lemma that
$A_{\text{ext}} \subset \partial_Q \Phi$.
So it remains to prove that $A_{\text{ext}}$ is maximal. To 
this aim, we observe that the operator
\begin{equation}\label{ZO}
  Z_\Omega:V \to V^*, \quad
   (Z_\Omega v)(x) := |v(x)|^{p(x)-2} v(x)
\end{equation}
is strictly monotone, bounded (cf.~Lemma \ref{L:dps-bdd} below),
 continuous, and coercive. The same happens, of course, for 
\begin{equation}\label{ZQ}
  Z_Q: \mathcal V \to \mathcal V^*, \quad
   (Z_Q v)(x,t) := |v(x,t)|^{p(x)-2} v(x,t).
\end{equation}
Here we also remark that 
$$
Z_\Omega(Pu(t)) = P (Z_Q(u))(t) \quad \mbox{ for } \ u \in \mathcal V.
$$
Properly modifying the proof of \cite[Chap.~II, Theorem 1.2,
 p.~39]{B76}, one can see that a (possibly) multivalued monotone
 graph $A$ in $V \times V^*$ (resp., $\mathcal V \times \mathcal V^*$)
 is maximal if and only if $A + \lambda Z_\Omega$ (resp., $A + \lambda Z_Q$) is
 surjective for some $\lambda>0$. In other words, one can use $Z_\Omega$
 (resp., $Z_Q$) in place of the \emph{duality mapping} between $V$ and $V^*$
 (resp., $\mathcal V$ and $\mathcal V^*$), which behaves badly with
 respect to integration in time. 
 Now, we are in position to prove that $A_{\text{ext}}$ is maximal in
 $\mathcal V \times \mathcal V^*$. Let $f \in \mathcal V^*$. Then, by
 Proposition \ref{P:ident}, $Pf$ belongs to
 $L^{(p^+)'}(0,T;V^*)$. Since $\partial_\Omega \varphi$ is maximal monotone
 in $V \times V^*$, one can uniquely take $u(t) \in D(\partial_\Omega \varphi)$
 such that
\begin{equation}\label{Zom}
 Z_\Omega \left(u(t)\right) + \partial_\Omega \varphi \left(u(t)\right) 
\ni Pf(t)
\quad \mbox{ for a.e. } t \in (0,T).
\end{equation}
The next task consists in proving that $u: t \mapsto u(t)$
is strongly measurable with values in $V$. 

Indeed, equation \eqref{Zom} can be rewritten as
$$
u(t) = \left(Z_\Omega + \partial_\Omega \varphi\right)^{-1} \left( Pf(t)
 \right).
$$
Let us now show that $T_\Omega := (Z_\Omega + \partial_\Omega
 \varphi)^{-1}$ is demicontinuous from $V^*$ into $V$. Indeed, let $g_n
 \to g$ strongly in $V^*$ and put $w_n := T_\Omega g_n$ and $w :=
 T_\Omega g$. Then, rewriting \eqref{Zom} with $u(t)$ and
 $Pf(t)$ replaced, respectively, by $w_n$ and $g_n$,
 and multiplying by $w_n$, we have
\begin{align}
\int_\Omega |w_n(x)|^{p(x)} \d x + \varphi(w_n) 
 &\leq \varphi(0) + \langle g_n, w_n \rangle_V \nonumber\\
 &\stackrel{\eqref{Young}}{\leq} \varphi(0) + \dfrac 1 2 \int_\Omega
 |w_n(x)|^{p(x)} \d x + C \int_\Omega |g_n(x)|^{p'(x)} \d x.
 \label{est-wg}
\end{align}
This relation, together with Proposition \ref{P:px-est}, ensures that $w_n$ is
 bounded in $V$ uniformly with $n$. 
 Hence, by subtraction of equations and multiplication by $w_n
 - w$, we get
$$
\left\langle Z_\Omega w_n - Z_\Omega w, w_n - w \right\rangle_V
\leq \langle g_n - g , w_n - w \rangle_V \to 0,
$$
which along with the definition of $Z_\Omega$ gives
$$
\left( |w_n(x)|^{p(x)-2}w_n(x) - |w(x)|^{p(x)-2}w(x) \right)
\left( w_n(x) - w(x) \right) \to 0
\quad \mbox{ for a.e. } x \in \Omega.
$$
Then, for almost all fixed $x \in \Omega$, $w_n(x)$ is uniformly bounded
and converges to $w(x)$. The combination of uniform boundedness in $V$
and pointwise convergence implies that (the whole sequence)
$w_n$ tends to $w$ weakly in $V$. Hence, $T_\Omega$ is demicontinuous 
from $V^*$ to $V$.

Now, let $h_n \in C([0,T];V^*)$ be such that $h_n \to Pf$ strongly in
 $L^{(p^+)'}(0,T;V^*)$. Then, from the demicontinuity of $T_\Omega$, we see
 that $u_n := T_\Omega(h_n(\cdot))$ is weakly continuous with values in
 $V$ on $[0,T]$, and, hence, $u_n$ is strongly measurable by Pettis'
 lemma and the separability of $V$. Moreover, $u_n(t)$ converges to
 $u(t) = T_\Omega(Pf(t))$ weakly in $V$ for a.e.~$t \in (0,T)$, and,
 therefore, $u$ is also strongly measurable in $(0,T)$  with values in
 $V$. Repeating an estimate similar to \eqref{est-wg}, one finds that
 \begin{equation}\label{uLpQ}
   \int^T_0 \left( \int_\Omega \left| \left(u(t)\right)(x)\right|^{p(x)}
	   \d x \right) \d t < \infty,
 \end{equation}
 which particularly implies $u \in L^{p^-}(0,T;V) \subset
 L^1(0,T;L^1(\Omega))$. Furthermore, we get $\hat u := P^{-1} u \in
 \mathcal V$ by Fubini's lemma along with \eqref{uLpQ}. Then, by
 \eqref{Zom}, $\hat u$ solves $Z_Q(\hat u) + A_{\text{ext}}(\hat u) \ni
 f$. In particular, $\hat u \in D(A_{\text{ext}})$ since 
 $f - Z_Q(\hat u) \in \mathcal V^*$. 
 Therefore, $A_{\text{ext}}$ is maximal monotone in
 $\mathcal V \times \mathcal V^*$. Since the graph of
 $A_{\text{ext}}$ is contained in that of a (maximal) monotone operator
 $\partial_Q \Phi$, the two operators must coincide, as desired.
\end{proof}


\section{Proof of Theorem \ref{T:ex}}\label{S:th1}

This section is aimed at giving a proof of Theorem \ref{T:ex}. As 
mentioned in the Introduction, it is a major difference 
of this study from the constant exponent case (e.g.,~\cite{Colli}) 
that one has to work in a mixed framework of (generalized) Lebesgue 
spaces of \emph{space and time variables} and of Lebesgue-Bochner
spaces (i.e., vector-valued Lebesgue spaces). Particularly, it is a
crucial point how to incorporate chain rules for subdifferentials into
such a specific framework. So let us begin with the following
proposition:
\begin{proposition}[Chain rule for subdifferentials in a mixed
 frame]\label{P:chain}
 Let $p(\cdot) \in \mathcal P(\Omega)$ satisfy $1 < p^- \leq
 p^+ < \infty$.
 Let $u \in \mathcal V := L^{p(x)}(Q)$ be such that $\partial_t u \in
 \mathcal V$. Suppose that there exists $\xi\in\mathcal V^*=L^{p'(x)}(Q)$
 such that $\xi\in \partial_Q \Phi(u)$, where $\Phi$ is given by
 \eqref{Phi} for a proper lower semicontinuous convex
 functional $\varphi$ on $V:=L^{p(x)}(\Omega)$. Then, the function $t
 \mapsto \varphi(Pu(t))$ is absolutely continuous over $[0,T]$.
 Moreover, for  each~$t\in(0,T)$, we have
$$
\dfrac{\d}{\d t} \varphi(Pu(t)) = 
\left\langle \eta,(P u)'(t) \right\rangle_V
\quad \mbox{ for all } \ \eta \in \partial_\Omega \varphi(Pu(t)),
$$
whenever $Pu$ and $\varphi(Pu(\cdot))$ are differentiable at $t$.
In particular, for $0 \leq s < t \leq T$, we have
$$
\varphi(Pu(t))-\varphi(Pu(s))
= \iint_{\Omega \times (s,t)} \xi \partial_\tau u ~\d x \, \d \tau.
$$
\end{proposition}

 This chain rule will be exploited at the end of the proof given
below for Theorem \ref{T:ex}, more precisely, for the identification of a limit
(see \S \ref{Ss:conv} for more details).


\subsection{Moreau-Yosida regularizations in variable exponent spaces}
\label{Ss:YM}

To prove Proposition \ref{P:chain}, let us introduce a variant of
\emph{Moreau-Yosida regularizations} (cf.~see~\cite{B76} for usual ones)
for functionals defined on variable exponent spaces. Namely, we set
\begin{equation}\label{philla}
  \varphi_\lambda(u):= \min_{v\in V}\left(
   \int_\Omega \frac \lambda {p(x)}
   \left|\dfrac{v(x)-u(x)}{\lambda}\right|^{p(x)} \,\d x
    + \varphi(v) \right)
    \quad \mbox{ for } \ u \in V,
\end{equation}
for $\varphi : V \to [0,\infty]$, and analogously,
\begin{equation}\label{Philla}
  \Phi_\lambda(u):= \min_{v\in \mathcal V}\left(
   \iint_Q \frac \lambda {p(x)}
   \left|\dfrac{v(x,t)-u(x,t)}{\lambda}\right|^{p(x)} \,\d x \, \d t
    + \Phi(v) \right)
    \quad \mbox{ for } \ u \in \mathcal V,
\end{equation}
for $\Phi : \mathcal V \to [0,\infty]$.
In the definition of $\varphi_\lambda$ and $\Phi_\lambda$, the position
of $\lambda$ is crucial, particularly in view of Lemma
\ref{L:Alam-bdd} below, due to the presence of variable
exponents. On the other hand, for constant exponent cases, the position
and the power of $\lambda$ in 
Moreau-Yosida regularizations would not be
a problem even in the $L^p$-framework, and, in fact, $\varphi_\lambda$ is
defined in a simpler way.

Moreover, define the \emph{modified resolvent} $J_\lambda : V \to V$ of
 $A := \partial_\Omega \varphi$ by setting $J_\lambda u := u_\lambda$,
 which is a unique solution of the equation
\begin{equation}\label{def-J}
Z_\Omega \left(\dfrac{u_\lambda - u}{\lambda}\right) 
+ A(u_\lambda) \ni 0 \quad \mbox{ in } \ V^*
\end{equation}
(here $Z_\Omega$ is defined as in Lemma \ref{lemmaext}). 
The \emph{modified Yosida approximation} $A_\lambda : V \to V^*$ of
$A$ is given by
$$
A_\lambda (u) := Z_\Omega\left(\dfrac{u - J_\lambda u}{\lambda}\right)
\in A(J_\lambda u)
\quad \mbox{ for each } \ u \in V.
$$
Then one can prove as in~\cite{B76} that $A_\lambda$ is single-valued
and monotone, and, furthermore, $J_\lambda$ and $A_\lambda$ are
demicontinuous. These notions and properties are still available for
general maximal monotone operators $A : V \to V^*$. Moreover,
analogue properties hold in the frame of $\mathcal V = L^{p(x)}(Q)$ 
as well.

Going back to the modified Moreau-Yosida regularization and following
the lines of \cite[Theorem~II.2.2, p.~57]{B76} with the proper
adaptations, we can also verify that $\varphi_\lambda$ is convex,
 continuous  and G\^ateaux differentiable in $V$. Moreover,
the subdifferential (= G\^ateaux derivative) $\partial_\Omega
\varphi_\lambda$ of $\varphi_\lambda$ coincides with the modified
Yosida approximation $A_\lambda = (\partial_\Omega \varphi)_\lambda$ of
$\partial_\Omega \varphi$.
Furthermore, the infimum in \eqref{philla} is achieved at $v = J_\lambda
u$, namely,
$$
\varphi_\lambda(u) = 
   \int_\Omega \frac \lambda {p(x)}
   \left|\dfrac{J_\lambda u(x)-u(x)}{\lambda}\right|^{p(x)} \,\d x
    + \varphi(J_\lambda u).
$$
Hence we have $D(\varphi_\lambda) = V$ and $\varphi(J_\lambda u) \leq
\varphi_\lambda (u) \leq \varphi(u)$ for any $u\in V$, which implies 
\begin{equation}\label{colla}
 \varphi_\lambda(u) \to \varphi(u) \quad \mbox{ for all } \  u\in V,
\end{equation}
since $J_\lambda u \to u$ strongly in $V$ for $u \in D(\varphi)$.
A further notable property is given by the following
\begin{lemma}\label{lemmayos}
 Noting as $\Phi_\lambda$ the Moreau-Yosida regularization of $\Phi$ 
 in $\mathcal V$ and as $\varphi_\lambda$ the Moreau-Yosida
 regularization of $\varphi$ in $V$, we have the relation
 \begin{equation}\label{phillaext}
   \Phi_\lambda(u) = \int^T_0 \varphi_\lambda(Pu(t))\,\d t
    \quad\mbox{ for all } \  u\in \mathcal V.
 \end{equation}
 In particular, for $u \in \mathcal V$ and $\xi \in \mathcal V^*$, {\rm
 Lemma~\ref{lemmaext}} ensures that 
$$
\xi_\lambda = \partial_Q \Phi_\lambda(u)
\quad \mbox{ if and only if } \quad 
P\xi_\lambda(t) = \partial_\Omega \varphi_\lambda(Pu(t))
 \ \mbox{ for almost all } t\in(0,T). 
$$
In other words, also in the variable exponent setting, Moreau-Yosida 
 regularization and integration in time commute. 
\end{lemma}
\begin{proof}
Let $u \in \mathcal V$. Then, we know that $Pu(t)\in V$ for
 a.e.~$t\in(0,T)$. Moreover, we also have
\begin{equation}\label{Philla1}
  \Phi_\lambda(u) = 
   \iint_Q \frac{\lambda}{p(x)}
   \left|\dfrac{u(t,x)-u_\lambda(t,x)}{\lambda}\right|^{p(x)} \,\d x \, \d t 
   + \Phi(u_\lambda),
\end{equation}
where $u_\lambda \in D(\partial_Q \Phi)$ satisfies
\begin{equation}\label{Philla2}
  Z_Q \left(\dfrac{u_\lambda-u}{\lambda}\right) + \partial_Q
   \Phi(u_\lambda) \ni 0.
\end{equation}
Analogously, since $Pu(t) \in V$ for a.a.~$t\in(0,T)$, there exists
 $\hat u_\lambda(t) := J_\lambda (P u(t)) \in D(\partial_\Omega
 \varphi)$, where $J_\lambda$ stands for the modified resolvent of
 $\partial_\Omega \varphi$, such that
\begin{equation}\label{philla1}
  \varphi_\lambda(Pu(t)) = 
   \int_\Omega \frac{\lambda}{p(x)}\left|\dfrac{u(x,t)-P^{-1}\hat u_\lambda(x,t)}{\lambda}\right|^{p(x)} \,\d x
   + \varphi(\hat u_\lambda(t)).
\end{equation}
Here $\hat u_\lambda(t)$ satisfies, for a.e.~$t\in (0,T)$,
\begin{equation}\label{philla2}
  Z_\Omega \left(\dfrac{\hat u_\lambda(t) - Pu(t)}{\lambda}\right) 
   + \partial_\Omega
   \varphi(\hat u_\lambda(t)) \ni 0.
\end{equation}
Since $J_\lambda : V \to V$ is demicontinuous, one can prove that $P^{-1}
 \hat u_\lambda \in \mathcal V$ proceeding as in Lemma \ref{lemmaext}.
Integrating \eqref{philla1} in time and using Fubini's lemma, we obtain
\begin{equation}\label{philla3}
  \int^T_0\varphi_\lambda(Pu(t))\,\d t = 
   \iint_Q \frac{\lambda}{p(x)} \left|\dfrac{u(t,x)-P^{-1}\hat
				 u_\lambda(t,x)}{\lambda} \right|^{p(x)}   \,\d x \, \d t
   + \Phi(P^{-1} \hat u_\lambda),
\end{equation}
provided that $\varphi(\hat u_\lambda(\cdot)) \in L^1(0,T)$.

Note that, up to this point, the functions $u_\lambda$ and $\hat
 u_\lambda$ need not be related to each other. However, observing that
$$
Z_\Omega \left( \dfrac{\hat u_\lambda(t)-Pu(t)}{\lambda}\right)
= Z_\Omega \left( P\left(\dfrac{P^{-1}\hat
 u_\lambda-u}{\lambda}\right)(t)\right)
= P \left( Z_Q\left(\dfrac{P^{-1}\hat
 u_\lambda-u}{\lambda}\right)\right)(t),
$$
by virtue of \eqref{philla2} and Lemma~\ref{lemmaext}, we infer that
\begin{equation}\label{philla2b}
 Z_Q\left( \dfrac{P^{-1} \hat u_\lambda - u}{\lambda}\right) 
  + \partial_Q \Phi(P^{-1} \hat u_\lambda) \ni 0.
\end{equation}
Comparing this with \eqref{Philla2}, we deduce that $P^{-1} \hat u_\lambda$
coincides with $u_\lambda$, whence we also obtain
$\varphi(\hat u_\lambda(\cdot)) =
 \varphi(Pu_\lambda(\cdot)) \in L^1(0,T)$ and the thesis follows from
 \eqref{philla3}.
\end{proof}


\subsection{Proof of Proposition \ref{P:chain}}\label{Ss:chain-proof}

Now, we are in position to prove Proposition \ref{P:chain}.
We first claim that
\begin{equation}\label{ipepalla}
  \int_s^t \left\langle P\xi_\lambda(\tau), (Pu)'(\tau)
	   \right\rangle_V \,\d \tau
   = \varphi_\lambda(Pu(t)) - \varphi_\lambda(Pu(s))
   \quad\mbox{ for all } \  s,t\in[0,T]
\end{equation}
for any function $u\in\mathcal V$ satisfying $\partial_t u\in \mathcal
 V$ and $\xi_\lambda=\partial_Q \Phi_\lambda(u)\in \mathcal V^*$ (see \S
 \ref{Ss:YM}). Indeed, we deduce from Proposition \ref{P:ident} that 
$$
  Pu\in W^{1,p^-}(0,T;V);
$$
hence, $Pu$ is absolutely continuous with values in $V$.
Furthermore, since $A_\lambda := \partial_\Omega \varphi_\lambda$ is
bounded from $V$ to $V^*$ and $P\xi_\lambda(t) = A_\lambda(Pu(t))$ by
Lemma \ref{lemmayos}, we find that 
$P\xi_\lambda$ belongs to $L^\infty(0,T;V^*)$.    
Therefore, by using a standard chain-rule for subdifferentials, one can
 obtain \eqref{ipepalla}.

We next pass to the limit as $\lambda\searrow 0$ in \eqref{ipepalla}.
Concerning the  left-hand  side, we first give the following lemma:

\begin{lemma}\label{L:Alam-bdd}
Let $A : V \to V^*$ be a maximal monotone operator and let $A_\lambda$
 be the modified Yosida approximation of $A$. Then for any $[u,\eta] \in
 A$ and $\lambda > 0$, it follows that
$$
\int_\Omega \dfrac 1 {p'(x)} |A_\lambda u(x)|^{p'(x)} \d x
\leq \int_\Omega \dfrac 1 {p'(x)} |\eta(x)|^{p'(x)} \d x.
$$
An analogous statement also holds for maximal monotone operators $\mathcal
 A : \mathcal V \to \mathcal V^*$.
\end{lemma}

In \S \ref{Ss:YM}, the notions of \emph{resolvent} and \emph{Yosida
approximation} (and hence, \emph{modified Moreau-Yosida regularization}
as well) were defined in such a way as to let this lemma hold true.
\begin{proof}
Let $[u,\eta] \in A$ and observe by the monotonicity of $A$ that
$0 \leq \langle \eta - A_\lambda(u), u - J_\lambda u \rangle_V$,
where $J_\lambda$ denotes the resolvent of $A$. By Young's inequality
 and the definition of $A_\lambda$,
 we have
\begin{align*}
\int_\Omega \left|\dfrac{u(x)-J_\lambda u(x)}{\lambda}\right|^{p(x)} \d x
&\leq \int_\Omega \eta(x) \dfrac{u(x)-J_\lambda u(x)}{\lambda} ~\d x\\
&\leq \int_\Omega \dfrac 1 {p'(x)} |\eta(x)|^{p'(x)} \d x
+ \int_\Omega \dfrac 1 {p(x)}
\left|\dfrac{u(x)-J_\lambda u(x)}{\lambda}\right|^{p(x)} \d x.
\end{align*}
Note that $|A_\lambda u(x)|^{p'(x)} = |Z_\Omega((u-J_\lambda
 u)/\lambda)(x)|^{p'(x)} = |(u(x)- (J_\lambda u)(x))/\lambda|^{p(x)}$. Then
 the desired inequality follows.
\end{proof}

By this lemma, we have
$$
 \iint_Q \dfrac 1 {p'(x)} |\xi_\lambda(x,t)|^{p'(x)} \d x \, \d t
 \leq  \iint_Q \dfrac 1 {p'(x)} |\xi(x,t)|^{p'(x)} \d x \, \d t,
$$
which along with Proposition \ref{P:px-est} implies the boundedness of
 $\xi_\lambda$ in $\mathcal V^*$.
Thus we deduce in particular that, for a subsequence
(not relabelled) of $\lambda\searrow 0$, 
$$
   \xi_\lambda \to \overline \xi \quad \mbox{weakly in }\, \mathcal V^*.
$$
This relation, together with the fact that $(Pu)' = P(\partial_t u)$ (see
   Proposition \ref{P:ident}), implies
\begin{align*}
\int^t_s \langle P \xi_\lambda(\tau), (Pu)'(\tau) \rangle_V \d \tau
&= \iint_{\Omega \times (s,t)} \xi_\lambda(x,\tau) \partial_t u(x,\tau) ~\d x \,
   \d \tau\\
&\to \iint_{\Omega \times (s,t)} \overline\xi(x,\tau) \partial_t u(x,\tau) ~\d x \, \d \tau
= \int^t_s \langle P \overline\xi(\tau), (Pu)'(\tau) \rangle_V \, \d \tau.
\end{align*}
Therefore, noting that $\varphi_\lambda(Pu(t)) \to \varphi(Pu(t))$ by
   \eqref{colla}, we have
$$
\varphi(Pu(t)) - \varphi(Pu(s)) 
=   \int_s^t \left\langle P\overline \xi(\tau), (Pu)'(\tau)
	   \right\rangle_V \,\d \tau
\quad
\mbox{ for all } \ 0 \leq s < t \leq T,
$$
which implies that $t \mapsto \varphi(Pu(t))$ is absolutely continuous on
   $[0,T]$, since $t \mapsto \langle P\overline \xi(t), (Pu)'(t)
  \rangle_V = \int_\Omega \overline\xi(x,t) \partial_t u(x,t) \d x$
   is integrable over $(0,T)$ by Fubini's lemma and the fact that
   $\overline\xi \partial_t u \in L^1(Q)$.

Now, let $t \in (0,T)$ be such that $Pu$ and $\varphi(Pu(\cdot))$ are
differentiable at $t$ and take $\eta \in \partial_\Omega \varphi(Pu(t))$
arbitrarily. Then by definition of subdifferential, we have, for $h > 0$,
$$
\dfrac{\varphi(Pu(t+h))-\varphi(Pu(t))}{h}
\geq \left\langle \eta , \dfrac{Pu(t+h)-Pu(t)}{h} \right\rangle_V.
$$
Taking the limit $h \to 0_+$ and using the differentiability of
$\varphi(Pu(\cdot))$, we obtain
$$
\dfrac{\d}{\d t} \varphi(Pu(t)) \geq \left\langle \eta, (Pu)'(t)
\right\rangle_V.
$$
The converse inequality also follows by choosing $h < 0$ and letting $h
\to 0_-$. Finally, we remark that $(Pu(t+h)-Pu(t))/h \to (Pu)'(t)$
strongly in $V$ for a.e.~$t\in(0,T)$, since $Pu$ belongs to
$W^{1,p^-}(0,T;V)$. In particular, substitute $\eta = P \xi(t) \in
\partial_\Omega \varphi(Pu(t))$ to get
$$
\varphi(Pu(t)) - \varphi(Pu(s))
= \int^t_s \left\langle P\xi(\tau), (Pu)'(\tau) \right\rangle_V \d \tau
= \iint_{\Omega \times(s,t)} \xi \partial_\tau u ~ \d x \, \d \tau.
$$
Thus we obtain the desired formula.

\subsection{Time-discretization}\label{Ss:time-disc}

We address system \eqref{pde}--\eqref{ic}
by means of the following time-discretization scheme:
\begin{equation}
 \partial_\Omega \psi \left( \dfrac{u_{n+1} - u_n}{h} \right)
 + \partial_\Omega \phi(u_{n+1}) = f_{n+1} \ \mbox{ in } V^*, \quad
 n = 0,1,\ldots, N-1,
\label{td}
\end{equation}
where $N \in \mathbb N$, $h := T/N$, $f_n \in V^*$ is given by
\begin{equation}\label{fn}
f_n := \dfrac 1 h \int^{t_n}_{t_{n-1}} Pf(\theta) \, \d \theta
= \dfrac 1 h \int^{t_n}_{t_{n_1}} f(\cdot,\theta) \, \d \theta
\end{equation}
 with $t_n := nh$ (hence $t_0 = 0$ and $t_N = T$) and the prescribed
 initial data $u_0$.
The existence of solutions $\{u_n\}_{n=1,2,\ldots, N}$ for \eqref{td}
can be proved at each step~$n$ by minimizing the functional
$J_n : V \to (-\infty,\infty]$ given by
$$
J_n(u) := h \psi \left( \dfrac{u-u_n}{h} \right)
+ \phi(u) - \langle f_{n+1}, u \rangle_V
\quad \mbox{ for } \ u \in V,
$$
which is convex and lower semicontinuous in $V$. Indeed, $J_n$ is
coercive in $V$ by the fact that
\begin{align*}
J_n(u) &\geq
h \psi \left( \dfrac{u-u_n}{h} \right) - C_n \|u\|_V\\
&= h \int_\Omega \dfrac{1}{p(x)} \left| \dfrac{u(x)-u_n(x)}{h} \right|^{p(x)} \d x
- C_n \|u\|_V\\
&\geq \dfrac{h}{p^+} \sigma_{p(\cdot)}^- \left( \left\|\dfrac{u-u_n}{h}\right\|_V \right)
- C_n \|u\|_V,
\end{align*}
which is coercive in $V$ by $p^- > 1$ for each fixed $n$.
Therefore, for each $n \in \{0,1,\ldots, N-1\}$, one can take a minimizer
$u_{n+1} \in D(\phi)$ of $J_n$. Hence, it holds that
$$
\partial_\Omega J_n(u_{n+1}) \ni 0 \ \mbox{ in } V^*.
$$
Since $D(\psi) = V$, by the sum rule for subdifferentials (see, e.g.,
~\cite{B76}), we have the representation formula
$$
\partial_\Omega J_n(u_{n+1}) =  \partial_\Omega \psi \left( \dfrac{u_{n+1} - u_n}{h} \right)
 + \partial_\Omega \phi(u_{n+1}) - f_{n+1}.
$$
Thus the minimizers $\{u_n\}_{n=1,2,\ldots,N}$ solve \eqref{td}.

We next introduce interpolants of the minimizers defined by
\begin{equation}\label{lint}
u_N(t) := \dfrac{t_{n+1}-t}{h} u_n + \dfrac{t-t_n}{h} u_{n+1}
\end{equation}
and
\begin{equation}\label{cint}
\overline u_N(t) \equiv u_{n+1}
\end{equation}
for $t \in (t_n,t_{n+1}]$ and $n = 0,1,\ldots, N-1$.
Then $u_N \in W^{1,\infty}(0,T;V)$ and $\overline u_N \in
L^\infty(0,T;V)$, and they satisfy
\begin{equation}\label{td2}
 \partial_\Omega \psi(u_N'(t)) + \partial_\Omega \phi(\overline u_N(t)) = \overline
  f_N(t)
  \ \mbox{ in }V^*, \quad \mbox{for a.e.~} \
  t \in (0,T), \quad u_N(0) = u_0,
\end{equation}
 where $\overline f_N$ is a piecewise constant interpolant of
$\{f_n\}_{n=0,1,\ldots,N}$ defined as in \eqref{cint}. 
Equivalently, by Proposition \ref{P:ident} and Lemma \ref{lemmaext},
\begin{equation}\label{td2-2}
 \partial_Q \Psi(\partial_t (P^{-1}u_N) ) + \partial_Q
  \Phi(P^{-1}\overline u_N) = P^{-1}\overline f_N
  \ \mbox{ in } \mathcal V^*, \quad u_N(0) = u_0.
\end{equation}
By convexity, we also find that $u_N(t)$ and $\overline u_N(t)$
belong to $D(\phi)$ for all $t \in [0,T]$.


\subsection{Lemmas for a priori estimates}\label{Ss:lemmas}

We first give a lemma on the boundedness and the convergence of the
piecewise constant interpolant $\overline f_N$ of
$\{f_n\}_{n=0,1,\ldots,N}$.
\begin{lemma}\label{L:bfN}
 It holds that
\begin{equation}\label{fN-est}
\iint_Q |P^{-1} \overline f_N|^{p'(x)} \d x \, \d t
\leq \iint_Q |f|^{p'(x)} \d x \, \d t.
\end{equation}
Moreover, $P^{-1} \overline f_N \to f$ strongly in $L^{p'(x)}(Q)$ as $N \to \infty$.
\end{lemma}

\begin{proof}
By using Jensen's inequality, we have, for all $t \in (t_{n-1}, t_n)$,
\begin{align*}
  |P^{-1} \overline f_N(x,t)|^{p'(x)} 
 &= \left| \dfrac 1 h \int^{t_n}_{t_{n-1}} f(x,t) \d t \right|^{p'(x)}
 \leq \dfrac 1 h \int^{t_n}_{t_{n-1}} |f(x,t)|^{p'(x)} \d t,
\end{align*}
which implies \eqref{fN-est}.
The convergence of $P^{-1}\overline f_N$ can be proved in a standard way
 (see, e.g.,~\cite[Lemma 8.7, p.~208]{Roubicek}, where a similar argument
 is performed in a Lebesgue-Bochner space setting). However, for the
 convenience of the reader, let us give a brief sketch of proof.
 Let $\vep > 0$ be any small number. Then one can take a smooth
 approximation $f_\vep \in C^\infty_0(Q)$ by $p(\cdot) \in \mathcal
 P(\Omega)$ and $p^+ < \infty$ (see~\cite[Theorem 3.4.12,
 p.~90]{DHHR11}) such that $\|f_\vep - f\|_{L^{p'(x)}(Q)} < \vep / 3$ and 
$\| P^{-1} \overline f_N - P^{-1} \overline{(f_\vep)}_N
 \|_{L^{p'(x)}(Q)} < \vep / 3$, where $\overline{(f_\vep)}_N \in
 L^\infty(0,T;V^*)$ denotes a piecewise constant interpolant for
 $f_\vep$ given in a similar way to $\overline f_N$ for $f$.
Indeed, the latter inequality can be checked from the former one by
 observing that, thanks to \eqref{fN-est},
$$
\iint_Q \left| P^{-1} \overline f_N - P^{-1} \overline{(f_\vep)}_N
 \right|^{p'(x)} \d x \, \d t = 
\iint_Q \left| P^{-1} \overline{(f - f_\vep)}_N \right|^{p'(x)} \d x \, \d t
\leq \iint_Q \left| f - f_\vep \right|^{p'(x)} \d x \, \d t
$$
and using Proposition \ref{P:px-est}.

On the other hand, since $f_\vep$ is uniformly continuous on $\overline
 Q$, there exists a modulus of continuity $\omega_\vep$ for $f_\vep$. 
 Then by Proposition \ref{P:Holder} it holds that
 \begin{align*}
 \left\|P^{-1}\overline{(f_\vep)}_N - f_\vep\right\|_{L^{p'(x)}(Q)} 
 &\leq 
  2 \| 1 \|_{L^{p'(x)}(Q)} \left\|P^{-1}\overline{(f_\vep)}_N - f_\vep\right\|_{L^\infty(Q)} \\
 &\leq 2 (|Q| + 1) \omega_\vep (h) \to 0
\end{align*}
 as $h \to 0$ (equivalently, $N \to \infty$). Here we also used
 $\| 1 \|_{L^{p'(x)}(Q)} \leq (|Q|+1)$. Actually,
 we have $\iint_Q (1/\lambda)^{p'(x)}
 \d x \, \d t \leq 1$ with $\lambda = |Q|+1$.
 Therefore, we can take $N_\vep \in \mathbb N$ such that
 $\|P^{-1}\overline{(f_\vep)}_N - f_\vep\|_{L^{p'(x)}(Q)} < \vep/3$ for
 any $N \geq N_\vep$. Consequently, it holds that $\|P^{-1}\overline f_N -
 f\|_{L^{p'(x)}(Q)} < \vep$ for any $N \geq N_\vep$.
\end{proof}

The following lemma provides continuous embeddings between variable
exponent Lebesgue spaces and Lebesgue-Bochner spaces through the
mappings $P, P^{-1}$.
\begin{lemma}\label{L:rel-sp}
The following {\rm (i)} and {\rm (ii)} are satisfied\/{\rm :}
\begin{enumerate}
 \item $\|P^{-1}u\|_{L^{p(x)}(Q)} \leq C \|u\|_{L^{p^+}(0,T;V)}$ \quad 
       for all \ $u \in L^{p^+}(0,T;V)$.
 \item $\|Pu\|_{L^{p^-}(0,T;V)} \leq C \|u\|_{L^{p(x)}(Q)}$ \quad
       for all \ $u \in L^{p(x)}(Q)$.
\end{enumerate}
\end{lemma}
\begin{proof}
By Proposition \ref{P:ident}, each $u \in L^{p^+}(0,T;V)
 \subset L^1(0,T;L^1(\Omega))$ has a unique representative $P^{-1} u \in
 L^1(Q)$. Then (i) follows. Indeed, for each $u \in L^{p^+}(0,T;V)$, by
 Fubini's lemma and Proposition \ref{P:px-est} we get
$$
\iint_Q |P^{-1}u(x,t)|^{p(x)} \d x \, \d t
\leq \int^T_0 \left( \|u(t)\|_V + 1\right)^{p^+} \d t,
$$
which implies $P^{-1}u \in L^{p(x)}(Q)$. Using Proposition \ref{P:px-est}
 again, we also note that
$$
\sigma_{p(\cdot)}^- \left( \|P^{-1}u\|_{L^{p(x)}(Q)} \right)
\leq
\iint_Q |P^{-1}u(x,t)|^{p(x)} \d x \, \d t.
$$
Therefore (i) follows.
As for (ii), see Proposition \ref{P:LB} for a proof.
\end{proof}

We next derive the boundedness of $\partial_Q \Psi : \mathcal V \to
\mathcal V^*$.

\begin{lemma}\label{L:dps-bdd}
 It holds that
\begin{equation}
 \left\|
  |v|^{p(x)-2} v
 \right\|_{\mathcal V^*} 
 \leq \left(
       \|v\|_{\mathcal V} + 1
      \right)^{p^+-1}
 \quad \mbox{ for all } \ v \in \mathcal V.
\end{equation}
\end{lemma} 

\begin{proof}
Set $\lambda = (\|v\|_{\mathcal V} + 1)^{p^+-1} \geq 1$.
Let us estimate the following:
$$
\iint_Q \left| \dfrac{|v|^{p(x)-2} v}{\lambda} \right|^{p'(x)} \d x \, \d t
= \iint_Q \dfrac{|v|^{p(x)}}{\lambda^{p'(x)}} \, \d x \, \d t.
$$
Then we note  that 
$$
\lambda^{p'(x)} = (\|v\|_{\mathcal V} + 1)^{p'(x)(p^+-1)}
\geq (\|v\|_{\mathcal V} + 1)^{p(x)} 
\geq \|v\|_{\mathcal V}^{p(x)}. 
$$
Thus we get
$$
\iint_Q \left| \dfrac{|v|^{p(x)-2}v}{\lambda} \right|^{p'(x)} \d x \, \d t
\leq \iint_Q \left| \dfrac{v}{\|v\|_{\mathcal V}} 
\right|^{p(x)} 
\d x \, \d t \leq 1,
$$
which implies the assertion.
\end{proof}


\subsection{A priori estimates}\label{Ss:est}

We are now in position to derive a priori estimates.
Let us first test \eqref{td} by $(u_{n+1}-u_n)/h$ to get
\begin{align*}
 \int_\Omega \left| \dfrac{u_{n+1}-u_n}{h} \right|^{p(x)} \d x
 + \left\langle \partial_\Omega \phi(u_{n+1}), \dfrac{u_{n+1}-u_n}{h}
 \right\rangle_V = \left\langle f_{n+1} , \dfrac{u_{n+1}-u_n}{h}
 \right\rangle_V,
\end{align*}
which, together with Propositions \ref{P:Holder} and
\ref{P:px-est} and Inequality \eqref{Young}, implies 
\begin{align*}
\lefteqn{
\int_\Omega \left| \dfrac{u_{n+1}-u_n}{h} \right|^{p(x)} \d x
 + \dfrac{\phi(u_{n+1}) - \phi(u_n)}{h}}\\
&\leq C \int_\Omega |f_{n+1}|^{p'(x)} \d x
+ \dfrac 1 2 \int_\Omega \left| \dfrac{u_{n+1}-u_n}{h} \right|^{p(x)} \d x.
\end{align*}
Hence
\begin{equation}\label{ei01}
\dfrac 1 2 \int_\Omega \left| \dfrac{u_{n+1}-u_n}{h} \right|^{p(x)} \d x
 + \dfrac{\phi(u_{n+1}) - \phi(u_n)}{h}
\leq C \int_\Omega |f_{n+1}|^{p'(x)} \d x.
\end{equation}
Multiplying this by $h$ and summing it up from $n = 0$ to $m \in
\{0,1,\ldots,N-1\}$, we have
\begin{equation}\label{est0}
\dfrac 1 2 \sum_{n=0}^m h
\int_\Omega \left| \dfrac{u_{n+1}-u_n}{h} \right|^{p(x)} \d x
 + \phi(u_{m+1}) 
\leq \phi(u_0) 
+ C \iint_Q |P^{-1}\overline f_N|^{p'(x)} \d x \, \d t.
\end{equation}
Thus by Proposition \ref{P:ident} and Lemma \ref{L:bfN} we obtain
\begin{equation}\label{est1}
 \iint_Q \left| \partial_t (P^{-1}u_N) \right|^{p(x)} \d x \, \d t
 +  \sup_{t \in (0,T]} \phi(\overline u_N(t)) \leq C,
\end{equation}
which together with Proposition \ref{P:px-est} also gives
\begin{equation}\label{est2}
 \sup_{t \in (0,T]} |\overline u_N(t)|_X + \sup_{t \in [0,T]} |u_N(t)|_X
 \leq C.
\end{equation}
One can also deduce from \eqref{est1} and Proposition \ref{P:px-est} that
\begin{equation}\label{est3}
 \left\| \partial_t (P^{-1}u_N) \right\|_{\mathcal V} \leq C.
\end{equation}
Thus by \eqref{est3} along with Lemma \ref{L:dps-bdd}, one can get
\begin{equation}\label{est4}
 \left\| \partial_Q \Psi(\partial_t (P^{-1} u_N) ) \right\|_{\mathcal
  V^*} \leq C.
\end{equation}
By comparison of terms in \eqref{td2-2} together with the boundedness
of $P^{-1}\overline f_{N}$ in $\mathcal V^*$ (see Lemma \ref{L:bfN})
again, we get
\begin{equation}\label{est5}
 \left\| \partial_Q \Phi( P^{-1} \overline u_N ) \right\|_{\mathcal V^*} \leq C.
\end{equation}


\subsection{Convergence}\label{Ss:conv}

Recall that $X = W^{1,m(x)}_0(\Omega)$ is compactly embedded in
$V = L^{p(x)}(\Omega)$ by \eqref{H2}. Hence by virtue of
\eqref{est2} and \eqref{est3} together with Proposition \ref{P:ident}
(particularly, $\partial_t (P^{-1}u_N) = P^{-1}(u_N')$) and (ii) of
Lemma \ref{L:rel-sp},
Aubin-Lions's compactness lemma ensures the precompactness of $(u_N)$ in
$C([0,T];V)$. Thus we obtain
\begin{align}
 u_N \to u \quad &\mbox{ strongly in } C([0,T];V) \mbox{
 and weakly in } W^{1,p^-}(0,T;V), \label{c:du:Cp} \\
 & \mbox{ weakly star in } L^\infty(0,T;X),
\end{align}
with $u \in W^{1,p^-}(0,T;V) \cap L^\infty(0,T;X)$. Furthermore, for
each $t \in (t_n,t_{n+1})$, noting that
\begin{align*}
 \left| P^{-1}u_N(x,t) - P^{-1}\overline u_N(x,t) \right|
 &= \left|  \dfrac{t_{n+1}-t}{h} u_n(x) + \dfrac{t-t_n}{h} u_{n+1}(x) 
 - u_{n+1}(x) \right|\\
 &= (t_{n+1}-t) \left|  \dfrac{u_{n+1}(x)-u_n(x)}{h} \right|
 \leq h \left|  \dfrac{u_{n+1}(x)-u_n(x)}{h} \right|,
\end{align*}
one has
\begin{align*}
\int_\Omega |P^{-1}u_N(x,t) - P^{-1}\overline u_N(x,t)|^{p(x)} \d x
&\leq \int_\Omega h^{p(x)} \left|
\dfrac{u_{n+1}(x)-u_n(x)}{h}
\right|^{p(x)} \d x
\\
&\stackrel{\eqref{est0}}{\leq} 2 h^{p^- - 1} \left(
\phi(u_0) + C \iint_Q |f|^{p'(x)} \d x \, \d t
\right)
\\
& \to 0 \quad \mbox{ uniformly for } \ t \in (0,T)
\quad \mbox{ as } \ h \to 0.
\end{align*}
Here we also used Lemma \ref{L:bfN}.
Thus it follows that
$$
u_N - \overline u_N \to 0 \quad \mbox{ strongly in } L^\infty(0,T;V).
$$
Combining this relation with \eqref{c:du:Cp}, we obtain
\begin{equation}\label{c:u-bu}
 \overline u_N \to u \quad \mbox{ strongly in } L^\infty(0,T;V). 
\end{equation}
Furthermore, it also holds by \eqref{est2} that
\begin{equation}\label{c:bu}
  \overline u_N \to u \quad \mbox{ weakly star in } L^\infty(0,T;X),
\end{equation}
which, together with the fact that $u \in C([0,T];V)$ and $X
\hookrightarrow V$, implies $u \in C_w([0,T];X)$, the space of weakly
continuous functions with values in $X$.

By \eqref{c:u-bu} along with Lemma \ref{L:rel-sp}, we infer that
\begin{equation}\label{c:u-bu2}
 P^{-1}\overline u_N \to P^{-1}u \quad \mbox{ strongly in } \mathcal V. 
\end{equation}
We set
$$
\hat u := P^{-1}u \in \mathcal V.
$$
By \eqref{est3}--\eqref{est5},
one can also take $\xi, \eta \in \mathcal V^*$ such that
\begin{align}
 \partial_t (P^{-1}u_N) \to \partial_t \hat u \quad &\mbox{ weakly in
 } \mathcal V,
\label{c:du}\\
 \partial_Q \Phi(P^{-1}\overline u_N) \to \xi 
 \quad &\mbox{ weakly in } \mathcal V^*, \label{c:dPhi}\\
 \partial_Q \Psi(\partial_t (P^{-1}u_N)) \to \eta 
 \quad &\mbox{ weakly in } \mathcal V^*. \label{c:dPsi}
\end{align}
Hence, we have in particular $\partial_t \hat u \in \mathcal V$.
Moreover, thanks also to Lemma \ref{L:bfN}, 
we can take the limit $n \to \infty$ in \eqref{td2-2}
to obtain $\eta + \xi = f$ in $\mathcal V^*$.
By the maximal monotonicity in $\mathcal V \times \mathcal V^*$ of the
subdifferential operator $\partial_Q \Phi$, we derive $[\hat u,\xi] \in
\partial_Q \Phi$ from \eqref{c:dPhi} and \eqref{c:u-bu2}.

We finally claim that
\begin{equation}\label{cl1}
 [\partial_t \hat u,\eta] \in \partial_Q \Psi.
\end{equation}
We can prove this fact by using  the chain rule established in
Proposition \ref{P:chain} as well as  a monotonicity argument. 
Indeed,  a standard chain rule and  a simple calculation 
yield 
\begin{align*}
 \iint_Q \partial_Q \Psi \left(\partial_t (P^{-1}u_N)\right) 
 \partial_t (P^{-1}u_N)  ~ \d x \, \d t
 &\stackrel{\eqref{td2-2}}{=} \iint_Q \left( P^{-1} \overline f_N - \partial_Q
 \Phi(P^{-1}\overline u_N) \right) \partial_t (P^{-1}u_N) ~  \d x \, \d t\\
 &\leq \iint_Q \left(P^{-1} \overline f_N\right) \partial_t (P^{-1}u_N) ~ \d x \, \d t
 - \phi(u_N(T)) + \phi(u_0).
\end{align*}
Hence, by virtue of the strong convergence of $P^{-1}\overline f_N$ to $f$
in $\mathcal V^*$ (see Lemma \ref{L:bfN}), we infer  that 
\begin{align*}
\lefteqn{
 \limsup_{n \to \infty} 
 \iint_Q \partial_Q \Psi 
\left(\partial_t (P^{-1}u_N) \right) \partial_t (P^{-1} u_N) ~ \d x \, \d t 
}\\
&\leq
\iint_Q f \partial_t \hat u ~ \d x \, \d t - \phi(u(T)) + \phi(u_0),
\end{align*}
since $\liminf_{N \to \infty} \phi(u_N(T)) \geq \phi(u(T))$ by \eqref{c:du:Cp}.
Furthermore, recall that $\partial_t \hat u \in \mathcal V$, $\xi \in
\mathcal V^*$, $[\hat u,\xi] \in \partial_Q \Phi$ and exploit the chain
rule in Proposition \ref{P:chain} to deduce that $t \mapsto \phi(u(t))$
is absolutely continuous on $[0,T]$, and, moreover, it holds that
\begin{align*}
\limsup_{n \to \infty} 
 \left\langle \partial_Q \Psi \left(\partial_t (P^{-1}u_N)\right), \partial_t
(P^{-1} u_N) \right\rangle_{\mathcal V}
& \leq
\iint_Q f \partial_t \hat u ~ \d x \, \d t - \phi(P\hat u(T)) + \phi(P\hat u(0))\\
&= \left\langle f - \xi ,\partial_t \hat u \right\rangle_{\mathcal V}
 = \left\langle \eta, \partial_t \hat u \right\rangle_{\mathcal V},
\end{align*}
which together with the maximal monotonicity of $\partial_Q \Phi$ in
$\mathcal V \times \mathcal V^*$ implies that $[\partial_t \hat u, \eta] \in
\partial_Q \Phi$. Thus $\hat u$ is a strong solution of
\eqref{pde}--\eqref{ic}  (see \S \ref{Ss:space}).  This completes the proof.


\section{Proof of Theorem \ref{T:regu}}\label{S:th2}

In this section, we prove Theorem \ref{T:regu} under the additional
regularity assumption \eqref{f-regu}. To this end, we shall derive the
second energy inequality. Write \eqref{td} for the couple of indexes $n$
and $n-1$ (for $n \in \{1,2,\ldots,N-1\}$) and then take the
difference. It follows that
$$
\partial_\Omega \psi \left( \dfrac{u_{n+1}-u_n}{h}\right)
- \partial_\Omega \psi \left( \dfrac{u_n-u_{n-1}}{h}\right)
+ \partial_\Omega \phi(u_{n+1}) - \partial_\Omega \phi(u_n)
= f_{n+1} - f_n.
$$
By the monotonicity of $\partial_\Omega \phi$, the multiplication of the
above by $u_{n+1}-u_n$ yields
\begin{equation}
\left\langle \partial_\Omega \psi \left( \dfrac{u_{n+1}-u_n}{h}\right)
- \partial_\Omega \psi \left( \dfrac{u_n-u_{n-1}}{h}\right)
, u_{n+1} - u_n \right\rangle_V
 \leq \langle f_{n+1} - f_n , u_{n+1} - u_n \rangle_V.
\label{dexdu}
\end{equation}
Here the left-hand side can be transformed as follows:
\begin{align}
\lefteqn{
 \left\langle \partial_\Omega \psi \left( \dfrac{u_{n+1}-u_n}{h}\right)
- \partial_\Omega \psi \left( \dfrac{u_n-u_{n-1}}{h}\right)
, u_{n+1} - u_n \right\rangle_V
}\nonumber\\
&\geq h \left[ \psi^* \left(
 \partial_\Omega \psi \left( \dfrac{u_{n+1}-u_n}{h}\right)
\right)
-  \psi^* \left(
 \partial_\Omega \psi \left( \dfrac{u_n-u_{n-1}}{h}\right)
\right) \right],
\label{star-eq}
\end{align}
where we also used that $\partial_\Omega \psi^* = (\partial_\Omega \psi)^{-1}$.
Hence, multiplying \eqref{dexdu} by $(n-1)$ and using \eqref{star-eq}, we have
\begin{align*}
(n-1)h \left[ \psi^* \left(
 \partial_\Omega \psi \left( \dfrac{u_{n+1}-u_n}{h}\right)
\right)
-  \psi^* \left(
 \partial_\Omega \psi \left( \dfrac{u_n-u_{n-1}}{h}\right)
\right) \right]
\\
\leq (n-1) \langle f_{n+1} - f_n , u_{n+1} - u_n \rangle_V.
\end{align*}
Summing it up from $n = 2$ to $m \in \mathbb N$, we observe by a simple
calculation with $m \in \{3,4,\ldots, N-1\}$ that
\begin{align}
\lefteqn{
(m-1) h \psi^* \left(
 \partial_\Omega \psi \left( \dfrac{u_{m+1}-u_m}{h}\right)
\right)
} \nonumber \\
&\leq 
 \sum_{n=2}^m h \psi^* \left(
 \partial_\Omega \psi \left( \dfrac{u_n-u_{n-1}}{h}\right)
\right)
 + \sum_{n=2}^m (n-1)h^2 \left\langle 
 \dfrac{f_{n+1} - f_n}{h} , \dfrac{u_{n+1} - u_n}{h} \right\rangle_V.
\label{ee2}
\end{align}
Then, we have to control the right-hand side. 
Firstly we notice that
$$
 \sum_{n=2}^m h \psi^* \left(
 \partial_\Omega \psi \left( \dfrac{u_n-u_{n-1}}{h}\right)
\right)
\leq \int^T_0 \psi^*\left( \partial_\Omega \psi 
(u_N'(t)) \right) \d t \leq C
$$
from \eqref{est1} together with the fact that
$$
\psi^* \left( \partial_\Omega \psi(v) \right) 
= \int_\Omega \dfrac{1}{p'(x)} \left| |v(x)|^{p(x)-2}v(x) \right|^{p'(x)}
\d x
= \int_\Omega \dfrac{1}{p'(x)} |v(x)|^{p(x)} \d x
$$
for any $v \in V$. Moreover, the following lemma holds:
\begin{lemma}\label{L:e:sdf}
It holds that
\begin{align*}
\lefteqn{
\sum_{n=2}^m (n-1)h^2 \left\langle 
 \dfrac{f_{n+1} - f_n}{h} , \dfrac{u_{n+1} - u_n}{h} \right\rangle_V
}\\
&\leq 
\iint_Q \dfrac 2 {p'(x)} |t \partial_t f|^{p'(x)} \d x \,\d t
+ \iint_Q \dfrac 1 {p(x)} \left|\partial_t (P^{-1}u_N)\right|^{p(x)} \d x \, \d t.
\end{align*}
\end{lemma}

\begin{proof}
By a simple computation,
\begin{align*}
\lefteqn{
\sum_{n=2}^m (n-1)h^2 \left\langle 
 \dfrac{f_{n+1} - f_n}{h} , \dfrac{u_{n+1} - u_n}{h} \right\rangle_V
}\\
&\leq 
\sum_{n=2}^m h
\int_\Omega \dfrac 1  {p'(x)} \left| (n-1) \left(f_{n+1}(x)-f_n(x)\right)
 \right|^{p'(x)} \d x
+ \iint_Q \dfrac 1 {p(x)} \left| \partial_t (P^{-1}u_N)(x,t)
 \right|^{p(x)} \d x \, \d t
\end{align*}
for $t \in (t_n,t_{n+1})$.
Here we further observe, by Jensen's inequality, that
\begin{align}
\left| (n-1) \left(f_{n+1}(x)-f_n(x)\right)
 \right|^{p'(x)}
%
&= \left|
\dfrac{1}{h} \int^{t_{n+1}}_{t_n} (n-1)\left(
f(x,\theta) - f(x,\theta - h)
\right) \d \theta
\right|^{p'(x)}\nonumber\\
&\leq \dfrac 1 h 
\int^{t_{n+1}}_{t_n} \left|
(n-1)\left( f(x,\theta) - f(x,\theta - h) \right)
\right|^{p'(x)} \d \theta \nonumber\\
&= \dfrac 1 h 
\int^{t_{n+1}}_{t_n} \left|
\int^\theta_{\theta-h} (n-1) \partial_s f(x,s) \d s
\right|^{p'(x)} \d \theta \nonumber\\
&\leq \dfrac 1 {h^2} 
\int^{t_{n+1}}_{t_n} \int^\theta_{\theta-h} 
\left| (n-1) h \partial_s f(x,s) \right|^{p'(x)}
 \d s \, \d \theta.
\label{f-est}
\end{align}
Since $(n-1)h \leq s$ for any $s \in (\theta-h, \theta)$ and $\theta \in
 (t_n,t_{n+1})$, we infer that
\begin{align}
\dfrac 1 {h^2} 
\int^{t_{n+1}}_{t_n} \int^\theta_{\theta-h} 
\left| (n-1) h \partial_s f(x,s) \right|^{p'(x)}
 \d s \, \d \theta
&\leq
\dfrac 1 {h^2} 
\int^{t_{n+1}}_{t_n} \int^\theta_{\theta-h} 
\left| s \partial_s f(x,s) \right|^{p'(x)}
 \d s \, \d \theta \nonumber\\
&\leq \dfrac 1 h \int^{(n+1)h}_{(n-1)h} 
\left| s \partial_s f(x,s) \right|^{p'(x)}
 \d s.\label{df-est}
\end{align}
Collecting \eqref{f-est} and \eqref{df-est} and exploiting Fubini's
 lemma, we obtain
\begin{align*}
\lefteqn{
\sum_{n=2}^m h 
\int_\Omega \dfrac 1 {p'(x)} \left| (n-1) \left(f_{n+1}(x)-f_n(x)\right)
 \right|^{p'(x)} \d x}\\
&\leq
\sum_{n=2}^m  
\int_\Omega \dfrac 1 {p'(x)}
\left(
\int^{(n+1)h}_{(n-1)h} 
\left| s \partial_s f(x,s) \right|^{p'(x)}
  \d s
\right) \d x\\
&\leq
\iint_Q \dfrac 2 {p'(x)} \left| s \partial_s f(x,s) \right|^{p'(x)}
 \d x \, \d s.
\end{align*}
Thus the proof is completed.
\end{proof}

Let $\delta \in (0,T)$ be fixed. 
Then we claim that one can choose $N \in \mathbb N$ so large
(equivalently, $h>0$ so small) that
\begin{equation}\label{est-7}
\esssup_{t \in (\delta,T)} \psi^* \left(
\partial_\Omega \psi(u_N'(t))
\right) \leq C/\delta
\end{equation}
for some $C > 0$ independent of $N$ and $\delta$.
Indeed, we note that
\begin{align*}
 (m-1)h \psi^* \left( \partial_\Omega \psi 
\left( \dfrac{u_{m+1}-u_m}{h} \right) \right)
&\geq \dfrac{m-1}{m+1} t \psi^* \left( \partial_\Omega \psi
(u_N'(t)) \right)
\quad \mbox{ for all } \ t \in \left( mh, (m+1)h \right).
\end{align*}
Hence it follows from \eqref{ee2} together with Lemma \ref{L:e:sdf} and 
Proposition \ref{P:px-est} that
\begin{equation}\label{e:tpsi*}
 t \psi^* \left( \partial_\Omega \psi (u_N'(t)) \right)
\leq  \dfrac{m+1}{m-1} \ell
\left(
\|t \partial_t f\|_{\mathcal V^*}
+ \left\| \partial_t (P^{-1}u_N) \right\|_{\mathcal V}
\right)
\end{equation}
with a non-decreasing function $\ell$ in $\mathbb R$
for all $t \in \left( mh, (m+1)h \right)$ and $m \in \{3,4,\ldots, N-1\}$.

Now, let $\delta \in (0,T)$ be fixed. For $h \in (0,\delta/3]$, one can
take $m_\delta \in \{3,4,\ldots,N-1\}$ such that $m_\delta h \leq \delta
\leq (m_\delta + 1) h$, which implies
$$
\dfrac{\delta}{h} - 1 \leq m_\delta \leq \dfrac{\delta}{h}
\quad \mbox{ and } \quad
\dfrac{m_\delta+1}{m_\delta-1} \leq \dfrac{\delta/h + 1}{\delta/h - 2}
\to 1 \quad \mbox{ as } \ h \to 0.
$$
Moreover, due to the monotonicity of the function $r \mapsto
(r+1)/(r-1)$, we note that
$$
\dfrac{m+1}{m-1} \leq
\dfrac{m_\delta+1}{m_\delta-1}
\quad \mbox{ for any } \ m \geq m_\delta.
$$
Hence, observing that $m_\delta h \leq \delta$, we conclude that
\begin{align*}
\esssup_{t \in (\delta,T)} t \psi^* \left(\partial_\Omega \psi(u_N'(t)) \right)
&\leq \esssup_{t \in (m_\delta h,T)} t \psi^* \left(\partial_\Omega
\psi( u_N'(t) ) \right)
\\
&\stackrel{\eqref{e:tpsi*}}{\leq} 
2 \ell \left(
\|t \partial_t f\|_{\mathcal V^*}
 + \left\|\partial_t (P^{-1}u_N) \right\|_{\mathcal V}
\right)
\end{align*}
for $0 < h \ll 1$.
By \eqref{est3} and assumption \eqref{f-regu}, we obtain
\eqref{est-7}. Moreover, we also infer that
\begin{equation}
\sup_{t \in [\delta,T]} \left\| \partial_\Omega \psi \left( u_N'(t) \right)
 \right\|_{V^*} \leq \dfrac{C}{\delta}.\label{est8}
\end{equation}
By a comparison of terms in \eqref{td2} along with Lemma
\ref{L:tf-bdd} given below, we further get
\begin{equation}
 \sup_{t \in [\delta,T]} \left\| \partial_\Omega \phi \left( \overline u_N(t) \right) \right\|_{V^*}\label{est9}
\leq \dfrac{C}{\delta}.
\end{equation}

\begin{lemma}\label{L:tf-bdd}
 Let $f \in L^{p'(x)}(Q)$ and assume \eqref{f-regu}. Then it follows that
$$
\sup_{t \in [0,T]} \|t \overline f_N(t)\|_{V^*} 
\leq \ell \left( \iint_Q |f|^{p'(x)} \d x \, \d t
+ \iint_Q |t \partial_t f|^{p'(x)} \d x \, \d t \right)
$$
with a nondecreasing function $\ell$ in $\mathbb R$.
\end{lemma}

\begin{proof}
For $t \in [t_n, t_{n+1})$, we see by Jensen's inequality that
\begin{align*}
\lefteqn{
 \int_\Omega |t P^{-1} \overline f_N(t,x)|^{p'(x)} \d x
}\\
&\leq \int_\Omega |t_{n+1} f_{n+1}(x)|^{p'(x)} \d x\\
&= \int_\Omega \left| \dfrac{t_{n+1}}{h} \int^{t_{n+1}}_{t_n} f(x,\tau)\,
 \d \tau \right|^{p'(x)} \d x\\
&= \int_\Omega \left| \dfrac{t_{n+1}}{h} \int^{t_{n+1}}_{t_n} \dfrac 1 \tau 
 P^{-1} \left\{ \int^{\tau}_0 \dfrac{\d}{\d s} \left(s Pf \right) \d s
 \right\} \d \tau \right|^{p'(x)} \d x\\
&\leq \sigma_{p(\cdot)}^+\left(\dfrac{t_{n+1}}{t_n} T\right)
 \dfrac{2^{(p^-)'-1}}{T}
\iint_Q \left( |f(x,s)|^{p'(x)} + \left| s \partial_s f(x,s)
 \right|^{p'(x)} \right) \d x \, \d s.
\end{align*} 
Here we used the facts that $sPf(s)|_{s = 0} = 0$, which will be checked
 below, and that $\partial_s(sf) = f + s \partial_s f$. Consequently,
 one can derive the desired inequality by 
 noting that $t_{n+1}/t_n \leq 2$.

We finally prove that $tPf(t)|_{t = 0} = 0$. By virtue of \eqref{f-regu}, we
 find that $tPf$ belongs to $W^{1,(p^+)'}(0,T;V^*)$ by $\partial_t(tf) = f
 + t \partial_t f \in \mathcal V^*$. Hence
 $tPf$ is continuous with values in $V^*$ on $[0,T]$. In particular, $t
 Pf(t)$ converges to some $g_0 \in V^*$ strongly in $V^*$ as $t \to 0_+$.
 Hence,
$$
\left\|tPf(t)\right\|_{V^*} \geq \dfrac 1 2 \|g_0\|_{V^*}
\quad \mbox{ for } \ 0 < t \ll 1,
$$
which implies $\left\|Pf(t)\right\|_{V^*} \geq \dfrac 1 {2t}
 \|g_0\|_{V^*}$. Since $\|Pf(\cdot)\|_{V^*} \in L^1(0,T)$, we conclude
 that $g_0 = 0$.
\end{proof}

By \eqref{est8} and \eqref{est9}, for any $\delta \in (0,T)$,
up to a subsequence, it holds that
\begin{align}
 \partial_\Omega \phi(\overline u_N(t))|_{(\delta,T)} \to P\xi|_{(\delta,T)} 
 \quad &\mbox{ weakly star in } L^\infty(\delta,T; V^*),\\
 \partial_\Omega \psi(u_N'(t))|_{(\delta,T)} \to P\eta|_{(\delta,T)} 
 \quad &\mbox{ weakly star in } L^\infty(\delta,T; V^*).
\end{align}
In particular, $P\xi, P\eta \in L^\infty(\delta,T ; V^*)$ for any
$\delta \in (0,T)$. Moreover, \eqref{est-7} also yields  that  $u' \in
L^\infty(\delta,T;V)$ for any $\delta \in (0,T)$.


\appendix

\section{Identification between Lebesgue and Bochner spaces}\label{S:LB}

We often identify the Lebesgue-Bochner space $L^p(0,T;L^p(\Omega))$ with
the Lebesgue space $L^p(Q)$ for $Q = \Omega \times (0,T)$.
For instance, a function $u = u(x,t) \in
L^p(Q)$ with space-time variables $(x,t) \in Q = \Omega \times (0,T)$
corresponds to an $L^p(\Omega)$-valued function $Pu \in
L^p(0,T;L^p(\Omega))$ through the mapping $P : L^p(Q) \to
L^p(0,T;L^p(\Omega))$ given by
\begin{equation}\label{Apdx:P}
(Pu) (t) := u(\cdot,t). 
\end{equation}
This mapping is well-defined and turns out to be linear, bijective and
isometric (see Proposition \ref{P:LB}).
Conversely, one may expect that each $L^p(\Omega)$-valued function $u =
u(t) \in L^p(0,T;L^p(\Omega))$ could be identified with a function
$Mu \in L^p(Q)$, where $M$ is defined by the relation
\begin{equation}\label{Apdx:M}
(Mu)(x,t) := \left( u(t) \right)(x).
\end{equation}
However, checking that $M$ is well-defined with values in $L^p(Q)$
is somehow delicate due to the different \emph{measures} 
characterizing the domain and the target space of the map.
For instance, it is known that $L^\infty(Q)$ 
does not coincide with $L^\infty(0,T;L^\infty(\Omega))$
(see, e.g.,~\cite[Example 1.42, p.~24]{Roubicek}) because of the
difference between the \emph{Lebesgue measurability} of functions in $Q$
and the \emph{strong measurability} of $L^p(\Omega)$-valued functions
over $(0,T)$ in Bochner's sense  (and also a lack of Pettis' theorem
for $L^\infty(\Omega)$). 

On the other hand, for $1 \leq p < \infty$, 
the two classes of spaces can be rigorously identified.
In this section, we revise the measure-theoretic 
arguments leading to this identification and 
show that, with a limited effort, also the case of 
(measurable) variable exponents can be covered.
The exposition will mainly follow the lines of the 
standard theory of vector-valued $L^p$-spaces,
so we do not claim any particular originality here.
However, the extension to the variable exponent 
case, although not difficult, seems to be new
and this is the main reason why we decided to 
include this part. Throughout this section, we denote by 
$\mathcal L^N$ the $N$-dimensional
Lebesgue measure.
\begin{remark}\label{R:meas}
{\rm
To be precise, the \emph{equivalence} in Lebesgue space $L^p(Q)$ and
 Lebesgue-Bochner space $L^p(0,T;L^p(\Omega))$ is to be interpreted
 as follows:
\begin{enumerate}
 \item[(i)] The equivalence ``$u_1 = u_2$ in $L^p(Q)$'' means that there
       exists an $\mathcal L^{N+1}$-measurable subset $\hat Q$ of $Q$
       such that $u_1(x,t) = u_2(x,t)$ for all $(x,t) \in \hat Q$ and
       $\mathcal L^{N+1}(Q \setminus \hat Q) = 0$. 
 \item[(ii)] On the other hand,
       ``$u_1 = u_2$ in $L^p(0,T;L^p(\Omega))$'' has to be intended 
       as $u_1(t) = u_2(t)$ in $L^p(\Omega)$ 
       for all $t \in I$, where $I$ is an $\mathcal
       L^1$-measurable subset of $(0,T)$ satisfying $\mathcal
       L^1((0,T)\setminus I)=0$. Moreover, ``$u_1(t) =
       u_2(t)$ in $L^p(\Omega)$'' means that $(u_1(t))(x) = (u_2(t))(x)$
       for all $x \in \Omega_t$ with some $\mathcal L^N$-measurable
       subset $\Omega_t$ of $\Omega$ satisfying $\mathcal L^N(\Omega \setminus
       \Omega_t) = 0$.
\end{enumerate}
}
\end{remark}
We begin with the following lemma regarding
the behavior of the operator $M$ in the class of simple functions.
\begin{lemma}\label{L:simple}
 The operator $M$ given by \eqref{Apdx:M} is well-defined for the class
 of simple functions with values in $L^p(\Omega)$. Moreover, it holds
 that $M = P^{-1}$ in that class.
\end{lemma}

\begin{proof}
Let $v : (0,T) \to L^p(\Omega)$ be given by
\begin{equation}\label{simple}
v(t) = \sum_{j \in J} f_j \chi_{I_j}(t),
\end{equation}
with a finite set $J$, disjoint subintervals $\{I_j\}_{j \in J}$ each
with positive measure, characteristics functions $\chi_{I_j}$ over $I_j$,
 and $\{f_j\}_{j \in J} \subset L^p(\Omega) \setminus \{0\}$.
Define $\hat v : Q \to \mathbb R$ by
$$
\hat v(x,t) := \sum_{j \in J} f_j(x) \chi_{I_j}(t).
$$
Then $\hat v$ is Lebesgue-measurable over $Q$. Indeed, for any $a \in
 \mathbb R$, the set
$$
\{(x,t) \in Q \colon \hat v \leq a\}
= \bigcup_{j \in J} \{x \in \Omega \colon f_j(x) \leq a\}
\times I_j
$$
is Lebesgue-measurable in $Q$, since so are $\{x \in \Omega \colon f_j(x)
 \leq a\}$ and $I_j$ in $\Omega$ and $(0,T)$, respectively. 

Even if $v$ is identified with some other simple function $v_0 : (0,T)
 \to L^p(\Omega)$ as
 in (ii) of Remark \ref{R:meas}, $\hat v$ coincides with $\hat v_0 :=
 (v_0(t))(x)$ a.e.~in $Q$. Actually, by assumption $\hat v(x,t) = \hat
 v_0(x,t)$ for all $t \in I$ and $x \in \Omega_t$ with some $I \subset
 (0,T)$ and a family $\Omega_t \subset \Omega$ with full measure. Set
$\hat Q := \{(x,t) \in Q \colon \hat v(x,t) = \hat v_0(x,t) \}$. Since
 $\hat v$ and $\hat v_0$ are Lebesgue measurable, $\hat Q$ is $\mathcal
 L^{N+1}$-measurable. Set $Z_t := \{x \in \Omega \colon (x,t) \in Q
 \setminus \hat Q\}$ for each $t \in (0,T)$. Then for all $t \in I$, we
 find that $Z_t \subset \Omega \setminus \Omega_t$; hence $\mathcal
 L^N(Z_t) = 0$. Hence by Fubini's lemma, we see that
$$
\mathcal L^{N+1}(Q \setminus \hat Q)
= \int^T_0 \mathcal L^N(Z_t) \, \d t
\leq \int_I \mathcal L^N(Z_t) \, \d t
+ \mathcal L^1((0,T)\setminus I) \mathcal L^N(\Omega) = 0.
$$
Hence $M$ is well-defined, and from the definition
we immediately get  the relation  $M = P^{-1}$.
This completes the proof.
\end{proof}

The following proposition enables us to identify $L^p(0,T;L^p(\Omega))$
with $L^p(Q)$.

\begin{proposition}[Relations between Lebesgue and Bochner spaces]\label{P:LB}
Let $\Omega$ be a {\rm (}possibly unbounded and non-smooth{\rm )} domain
 of $\mathbb R^N$. Let $Q = \Omega \times (0,T)$ with $T > 0$. 
\begin{enumerate}
 \item If $p^+ < \infty$, then the mapping $P$ {\rm (}given as in
       \eqref{Apdx:P}{\rm )} is well-defined from $L^{p(x)}(Q)$ into
       $L^{p^-}(0,T;L^{p(x)}(\Omega))$. Moreover, $P$ is linear,
       injective and continuous,  that is, there exists  a
       constant $C \geq 0$ such that
       $$
       \|Pu\|_{L^{p^-}(0,T;L^{p(x)}(\Omega))} \leq C
       \|u\|_{L^{p(x)}(Q)}
       \quad \mbox{ for all } \ u \in L^{p(x)}(Q).
       $$
       In particular, if $p(x) \equiv p < \infty$, then $P : L^p(Q) \to
       L^p(0,T;L^p(\Omega))$ is isometric, i.e.,
       $\|Pu\|_{L^p(0,T;L^p(\Omega))} = \|u\|_{L^p(Q)}$ for all $u \in L^p(Q)$.
 \item For each $u \in L^p(0,T;L^p(\Omega))$, there exists a
       unique representative $\hat u \in L^p(Q)$ of $u$, i.e.,
       $\hat u$ is the unique function in $L^p(Q)$ such that $u = P
       \hat u$.
       Define a mapping ${R} : L^p(0,T;L^p(\Omega)) \to L^p(Q)$ by ${R}u :=
       \hat u$. Then ${R}$ is linear, bijective and isometric.
       Furthermore, $P({R}u) = u$ for all $u \in L^p(0,T;L^p(\Omega))$ and
       ${R}(Pu) = u$ for all $u \in L^p(Q)$\/{\rm ;} hence, $R = P^{-1}$.
       It also holds that
       $$
       u(t) = Ru(\cdot,t) \quad \mbox{ for a.e. } t \in (0,T)
       $$
       for every $u \in L^p(0,T;L^p(\Omega))$.
\end{enumerate}
\end{proposition}

\begin{proof}
 We first verify (i). Let $u \in L^{p(x)}(Q)$. 
 Since $u$ is measurable and $(x,t) \mapsto |u(x,t)|^{p(x)}$ is
 integrable over $Q$, by Fubini's lemma it holds that $u(\cdot,t) \in
 L^{p(x)}(\Omega)$ for a.e.~$t \in (0,T)$. 
 Let us claim that $u(\cdot,t)$ can be uniquely
 determined in the sense of (ii) of Remark
 \ref{R:meas}, although $u$ has to be intended as 
 an equivalence class of
 $L^{p(x)}(Q)$. Indeed, let $u_1, u_2 \in L^{p(x)}(Q)$ satisfy $u_1 =
 u_2$ in $L^{p(x)}(Q)$, that is, $u_1(x,t) = u_2(x,t)$ for all $(x,t)
 \in \hat Q$ with some subset $\hat Q$ of $Q$ satisfying $\mathcal
 L^{N+1}(Q \setminus \hat Q) = 0$. Set $Z_t := \{x \in \Omega \colon
 (x,t) \in Q \setminus \hat Q\}$ for each $t \in (0,T)$. Then by
 Fubini's lemma, $\mathcal L^N(Z_t) = 0$ for all $t \in I$ with a subset
 $I \subset (0,T)$ satisfying $\mathcal L^1((0,T)\setminus I) = 0$. Then
 $u_1(\cdot,t) = u_2(\cdot,t)$ a.e.~in $\Omega$ for all $t \in I$. 
 Thus $u(\cdot,t)$ is uniquely determined as a vector-valued function. 

 Now, let us define an $L^{p(x)}(\Omega)$-valued function $Pu : (0,T) \to
 L^{p(x)}(\Omega)$ by $(Pu)(t) := u(\cdot,t)$.
 We next verify the strong measurability of $Pu$ in $(0,T)$. 
For any $v \in L^{p'(x)}(\Omega) = (L^{p(x)}(\Omega))^*$,
it follows by Proposition \ref{P:Holder} that $uv \in
 L^1(Q)$. 
Therefore we observe by Fubini's lemma that 
$$
\langle v , (Pu)(t) \rangle_{L^{p(x)}(\Omega)} = 
\int_\Omega v(x) u(x,t) ~\d x \in L^1(0,T).
$$
Hence $Pu : (0,T) \to L^{p(x)}(\Omega)$ is weakly measurable 
thanks to the arbitrariness of $v$. Since $L^{p(x)}(\Omega)$ is
 separable by $p^+ < \infty$, Pettis' theorem ensures that $Pu$ is also
 strongly measurable. Moreover, by Proposition \ref{P:px-est} it follows
 that
$$
\int^T_0 \sigma_{p(\cdot)}^- \left(\|(Pu)(t)\|_{L^{p(x)}(\Omega)}\right) \d t 
\leq \int^T_0 \left(
\int_\Omega |u(x,t)|^{p(x)} \d x
\right) \d t
= \iint_Q |u(x,t)|^{p(x)} \d x \, \d t < \infty.
$$
Let us define the (measurable) set
$\mathcal T:= \big\{ t\in (0,T):~ \| Pu(t)\|_{L^{p(x)}(\Omega)}\le 1 \big\}$.
Then it follows that
\begin{align*}
 \int^T_0 \|Pu(t)\|_{L^{p(x)}(\Omega)}^{p^-} \d t
 &= \int_{\mathcal T} \|Pu(t)\|_{L^{p(x)}(\Omega)}^{p^-} \d t
 + \int_{(0,T)\setminus\mathcal T} \sigma_{p(\cdot)}^- \left(
 \|Pu(t)\|_{L^{p(x)}(\Omega)} \right) \d t \\ 
 &\leq T + \iint_Q |u(x,t)|^{p(x)} \d x \, \d t < \infty,
\end{align*}
which implies $Pu \in L^{p^-}(0,T;L^{p(x)}(\Omega))$.
Hence, we obtain that $P$ is a well-defined 
operator from $L^{p(x)}(Q)$ into $L^{p^-}(0,T;L^{p(x)}(\Omega))$. 

The linearity of $P$ follows immediately from its definition.
Let us next check the injectivity of $P$. Let $u_1, u_2 \in L^{p(x)}(Q)$
 satisfy $Pu_1 = Pu_2$. Then one can take $I \subset (0,T)$ and a
 family $(\Omega_t)_{t \in I}$ as in (ii) of Remark \ref{R:meas} such that
 $((Pu_1)(t))(x) = ((Pu_2)(t))(x)$ for all $(x,t) \in \{(x,t)
 \colon x \in \Omega_t, \ t \in I \}$. Since $( (Pu_1)(t) )(x) - (
 (Pu_2)(t) )(x) = u_1(x,t) - u_2(x,t)$ are Lebesgue measurable over $Q$,
 the set $\hat Q = \{(x,t) \in Q \colon u_1(x,t) = u_2(x,t) \}$ is
 $\mathcal L^{N+1}$-measurable. Thus we infer that $u_1(x,t) =
 u_2(x,t)$ for all $(x,t) \in \hat Q$; moreover, by Fubini's lemma,
 the measure $\mathcal L^{N+1}(Q \setminus \hat Q)$ of the complement is
 zero, since $\mathcal L^1((0,T)\setminus I) = 0$ and $\mathcal
 L^N(\Omega \setminus \Omega_t) = 0$ for all $t \in I$. Hence $P$ is
 injective.

In particular, if $p(x) \equiv p$, then, by Fubini's lemma, we see that
$$
\|u\|_{L^p(Q)}^p
= \iint_Q |u(x,t)|^p \d x \, \d t
= \int^T_0 \left( \int_\Omega  |u(x,t)|^p \d x \right) \d t
= \|Pu\|_{L^p(0,T;L^p(\Omega))}^p
$$
for all $u \in L^p(Q)$.

We next prove (ii). For each $u \in L^p(0,T;L^p(\Omega))$, one can take
 a sequence $(u_n)$ of $L^p(\Omega)$-valued simple functions such that
\begin{equation}\label{1}
u_n \to u \quad \mbox{ strongly in } L^p(0,T;L^p(\Omega)).
\end{equation}
Hence $(u_n)$ forms a Cauchy sequence in $L^p(0,T;L^p(\Omega))$.
Recalling Lemma \ref{L:simple} and using Fubini's lemma, we have
\begin{align*}
\|M u_n - M u_m\|_{L^p(Q)}^p
&= \iint_Q \left| (Mu_n)(x,t)-(Mu_m)(x,t) \right|^p \d x \, \d t\\
&= \int^T_0 \left( \int_\Omega \left|(Mu_n)(x,t)-(Mu_m)(x,t)
 \right|^p \d x \right) \d t\\
&= \|u_n - u_m\|_{L^p(0,T;L^p(\Omega))}^p \to 0
\quad \mbox{ as } \ n,m \to \infty.
\end{align*}
Thus $(Mu_n)$ forms a Cauchy sequence in $L^p(Q)$. Hence $Mu_n$
 converges to some element $\hat u$ strongly in $L^p(Q)$. Moreover, we
 obtain that $u = P\hat u$ in $L^p(0,T;L^p(\Omega))$ by observing that 
\begin{align*}
\lefteqn{
\left( \int^T_0 \left\|u(t)-(P\hat u)(t)\right\|_{L^p(\Omega)}^p \d
 t \right)^{1/p}
}\\
&\leq \left( \int^T_0 \|u(t)- u_n(t)\|_{L^p(\Omega)}^p \d t \right)^{1/p} 
+  \left( \int^T_0 \left\|u_n(t)-(P\hat u)(t)\right\|_{L^p(\Omega)}^p \d
 t \right)^{1/p}\\
&= \left( \int^T_0 \|u(t)- u_n(t)\|_{L^p(\Omega)}^p \d t \right)^{1/p} 
+  \left( \iint_Q |(Mu_n)(x,t)-\hat u(x,t)|^p \d x \, \d t \right)^{1/p} \to 0
\end{align*}
as $n \to \infty$. Here we used the fact that $u_n=P(Mu_n)$.
Thus we get $u = P \hat u$, which also implies
$$
\|\hat u\|_{L^p(Q)} \stackrel{\text{(i)}}{=} \|P \hat u\|_{L^p(0,T;L^p(\Omega))}
= \|u\|_{L^p(0,T;L^p(\Omega))}.
$$
Furthermore, let $\hat u_1, \hat u_2 \in L^p(Q)$ be representatives
 of $u \in L^p(0,T;L^p(\Omega))$, that is, $P\hat u_1 = u = P\hat u_2$
 in $L^p(0,T;L^p(\Omega))$. Then it follows from (i) that $\|\hat u_1 -
 \hat u_2\|_{L^p(Q)} = \|P\hat u_1 - P\hat u_2\|_{L^p(0,T;L^p(\Omega))}
 = 0$, which implies $\hat u_1 = \hat u_2$ a.e.~in $Q$. Hence the
 representative $\hat u \in L^p(Q)$ of $u$ is uniquely determined.

Thus one can define the mapping ${R} : L^p(0,T;L^p(\Omega)) \to L^p(Q)$ by
 setting ${R}u = \hat u$ for each $u \in L^p(0,T;L^p(\Omega))$. Then, from the
 above, one can immediately find that $P(Ru) = u$ for all $u \in
 L^p(0,T;L^p(\Omega))$ (hence, ${R}$ is injective), and ${R}$ is
 isometric, that is,
\begin{equation}\label{Q-isom}
\|{R}u\|_{L^p(Q)} = \|u\|_{L^p(0,T;L^p(\Omega))}
\quad \mbox{ for all } \ u \in L^p(0,T;L^p(\Omega)).
\end{equation}
We also claim that ${R}$ is linear. Indeed, for $u_1, u_2 \in
 L^p(0,T;L^p(\Omega))$ and $\alpha, \beta \in \mathbb R$, let us take
 simple functions $u_{1,n}, u_{2,n} : (0,T) \to L^p(\Omega)$ such that
 $u_{i,n} \to u_i$ strongly in $L^p(0,T;L^p(\Omega))$ for $i =
 1,2$. Then observing by the linearity of $M$ that 
$$
{R}(\alpha u_1 + \beta u_2) :=
 L^p(Q)\text{-}\lim \left( M(\alpha u_{1,n} + \beta u_{2,n}) \right)
 = \alpha {R} u_1 + \beta {R} u_2,
$$
we conclude that ${R}$ is linear. 
We next claim that ${R}(Pu) = u$ for all $u \in L^p(Q)$.
Indeed, let $u \in L^p(Q)$. Then $Pu$ belongs to $L^p(0,T;L^p(\Omega))$ by (i);
 hence, there exist simple functions $(Pu)_n$ such that $(Pu)_n \to
 Pu$ strongly in $L^p(0,T;L^p(\Omega))$. One can prove that ${R}((Pu)_n)
 = M((Pu)_n)$. We observe by (i) and Lemma \ref{L:simple} that
 $M((Pu)_n) \to u$ strongly in $L^p(Q)$. On the other hand,
 ${R}((Pu)_n)$ converges to ${R}(Pu)$ strongly in $L^p(0,T;L^p(\Omega))$
 by \eqref{Q-isom}. Thus we get ${R}(Pu) = u$ for all $u \in
 L^p(Q)$. Furthermore, the procedure also yields that ${R}$ is surjective. 
 Thus we conclude that $R$ and $P = R^{-1}$ are bijective.

 Finally, recall that $P(Ru) = u$ for all $u \in L^p(0,T;L^p(\Omega))$ to
 obtain $u(t) = P(Ru)(t) = R u(\cdot,t)$ for a.e.~$t \in (0,T)$. The
 proof is completed.
\end{proof}

The next proposition states, roughly speaking, that 
the operators $R, P$ behave well with respect to
differentiation in time.
\begin{proposition}\label{P:tD}
 Let $\Omega$ be a {\rm (}possibly unbounded and nonsmooth{\rm )} domain
 of $\mathbb R^N$ and let $T > 0$. Then, for every $u \in
 W^{1,p}(0,T;L^p(\Omega))$, we have
$$
 \partial_t (Ru) = R(u') \in L^{p}(Q).
$$
Moreover, if $u \in L^{p(x)}(Q)$ and $\partial_t u \in L^{p(x)}(Q)$, then
 $Pu \in W^{1,p^-}(0,T;L^{p(x)}(\Omega))$ and
$$
(Pu)' = P(\partial_t u) \in L^{p^-}(0,T;L^{p(x)}(\Omega)).
$$
\end{proposition}

\begin{proof}
 Let $u \in W^{1,p}(0,T;L^p(\Omega))$. Then $u' \in
 L^p(0,T;L^p(\Omega))$ and 
$$
\int^T_0 u'(t) \rho(t) \, \d t = - \int^T_0 u(t) \rho'(t) \, \d t
\quad \forall \rho \in C^\infty_0(0,T)
$$
(see, e.g.,~\cite{B76} for vector-valued distributions).
Moreover, for any $\phi \in C^\infty_0(Q)$, we see that
$$
\int^T_0 \left\langle (P\phi)(t) , u'(t) \right\rangle_{L^p(\Omega)} \d t 
= - \int^T_0 \left\langle (P\phi)'(t) , u(t) \right\rangle_{L^p(\Omega)} \d t.
$$
Since $u$ and $u'$ belong to $L^p(0,T;L^p(\Omega))$, 
 using
 Proposition~\ref{P:LB} and the fact that $(P\phi)' = P(\partial_t
 \phi)$ (indeed, one can prove this just by integration by parts), 
 the above equality can be rewritten as
$$
\iint_Q \phi R(u') ~\d x \, \d t 
= - \iint_Q (Ru) \partial_t \phi ~\d x \, \d t
\quad \forall \phi \in C^\infty_0(Q),
$$
which implies $\partial_t (Ru) = R(u') \in L^p(Q)$.

If $u \in L^{p(x)}(Q)$ and $\partial_t u \in L^{p(x)}(Q)$, then
by standard integration by parts in Sobolev spaces we have
$$
\iint_Q \phi \partial_t u ~ \d x \, \d t
= - \iint_Q u \partial_t \phi~\d x \, \d t
\quad \mbox{ for all } \ \phi \in C^\infty_0(Q).
$$
Let $\rho \in C^\infty_0(0,T)$ and $w \in C^\infty_0(\Omega)$. Then,
 since $\phi(x,t) = \rho(t)w(x) \in C^\infty_0(Q)$, using Fubini's lemma
 we see that
\begin{align*}
\left\langle w, \int^T_0 Pu(t) \rho'(t) \d t \right\rangle_{L^{p(x)}(\Omega)}
&= \int^T_0 \langle \rho'(t) w , Pu(t) \rangle_{L^{p(x)}(\Omega)} \, \d t\\
&= \iint_Q \rho'(t) w(x) u(x,t)  ~\d x \, \d t\\
&= - \iint_Q \rho(t) w(x) \partial_t u(x,t)  ~\d x \, \d t\\
&= - \left\langle w, \int^T_0 \rho(t) P(\partial_t u)(t) \, \d t
 \right\rangle_{L^{p(x)}(\Omega)}.
\end{align*}
From the fact that $C^\infty_0(\Omega)$ is dense in $L^{p'(x)}(\Omega)$
 (see Theorem 3.4.12 of~\cite{DHHR11}), it follows that
$$
\int^T_0 Pu(t) \rho'(t) \, \d t
= - \int^T_0 \rho(t) P(\partial_t u)(t) \, \d t.
$$
Thus $(Pu)' = P(\partial_t u)$.
\end{proof}



\section{Lower semicontinuity of $\Phi$ on $L^{p(x)}(Q)$}
\label{app:B}

Let $u_n \in \mathcal V = L^{p(x)}(Q)$ be such that $\Phi(u_n) \leq a$
with an arbitrary $a \in \mathbb R$ and $u_n \to u$ strongly in $\mathcal
V$. Then by Fatou's lemma along with the fact that $\sup_{n \in \mathbb N}
\Phi(u_n) \leq a$, 
we see that $\liminf_{n \to \infty} \phi(Pu_n(\cdot)) \in L^1(0,T)$ and
\begin{equation}\label{fubini}
\int^T_0 \liminf_{n \to \infty} \phi(Pu_n(t)) ~\d t
\leq \liminf_{n \to \infty} \int^T_0 \phi(Pu_n(t)) ~\d t \leq a.
\end{equation}
By (ii) of Lemma \ref{L:rel-sp}, $Pu_n \to Pu$ strongly in
$L^1(0,T;V)$. Hence it follows from the lower semicontinuity of $\phi$
in $V$ that
\begin{equation}\label{lsc1}
\phi(Pu(t)) \leq \liminf_{n \to \infty} \phi(Pu_n(t)) < \infty
\quad \mbox{ for a.e. } t \in (0,T),
\end{equation}
which particularly means that $Pu(t) \in D(\phi)$ for a.e.~$t \in
(0,T)$.

On the other hand, we claim that $\phi(Pu(\cdot))$ is Lebesgue
measurable in $(0,T)$. Indeed, since $Pu$ is strongly measurable with
values in $V$ in $(0,T)$, one can take step functions $s_n : (0,T) \to
V$ such that $s_n (t) \to Pu(t)$ strongly in $V$ for a.e.~$t \in
(0,T)$. Let $\phi_\lambda$ be the Moreau-Yosida regularization of $\phi$
in $V$ (in a standard sense). Then $\phi_\lambda(s_n(\cdot))$ is also a
step function in $\Omega$; moreover, $\phi_\lambda(s_n(t)) \to
\phi_\lambda(Pu(t))$ for a.e.~$t \in (0,T)$. Hence
$\phi_\lambda(Pu(\cdot))$ is Lebesgue measurable in
$(0,T)$. Furthermore, since $\phi_\lambda(Pu(t)) \to \phi(Pu(t)) <
\infty$ for a.e.~$t \in (0,T)$, we deduce that $\phi(Pu(\cdot))$ is also
measurable.

Recall \eqref{lsc1} and the measurability of $\phi(Pu(\cdot))$ to get
$$
\int^T_0 \phi(Pu(t)) ~\d t \leq \int^T_0 \liminf_{n \to \infty}
\phi(Pu_n(t)) ~\d t.
$$
Thus it follows from \eqref{fubini} that
$$
\Phi(u) = \int^T_0 \phi(Pu(t)) ~\d t \leq \liminf_{n \to \infty}
\int^T_0 \phi(Pu_n(t)) ~\d t \leq a.
$$
Therefore, $\Phi$ is lower semicontinuous in $\mathcal V$. \qed


%

\end{document}